\documentclass[a4,reqno]{amsart}

\usepackage{t1enc}
\usepackage{amssymb}
\usepackage{graphicx}
\usepackage{comment}
\usepackage{mathrsfs}
\usepackage[size=small]{todonotes}

\usepackage{color}
\usepackage[hypertexnames=false,
    pdftex,
	pdfpagemode=UseNone,
	breaklinks=true,
	extension=pdf,
	colorlinks=true,
	linkcolor=blue,
	citecolor=blue,
	urlcolor=blue,
]{hyperref}

\usepackage[capitalize]{cleveref}

\newcommand{\cyclic}{\mathop{\kern0.9ex{{+}\kern-2.10ex\raise-0.20
      ex\hbox{\Large\hbox{$\circlearrowright$}}}}\limits}

\theoremstyle{definition}

\newtheorem{theorem}{Theorem}[section]
\newtheorem{proposition}[theorem]{Proposition}
\newtheorem{definition}[theorem]{Definition}
\newtheorem{lemma}[theorem]{Lemma}
\newtheorem{corollary}[theorem]{Corollary}

\newtheorem*{theorem*}{Theorem}
\newtheorem*{corollary*}{Corollary}

\newtheorem*{proposition*}{Proposition}

\newtheorem{conjecture}[theorem]{Conjecture}

\newtheorem{remark}[theorem]{Remark}

\newtheorem{example}[theorem]{Example}

\def\cC{\mathcal C }

\def\cE{\mathcal E}

\def\cP{\mathcal P}

\def\cO{\mathcal O}
\def\cS{\mathcal S }

\def\cK{\mathcal K}

\def\gf{{\mathfrak f}}

\newcommand{\C}{\mathbb{C}}

\renewcommand{\P}{\mathbb{P}}

\newcommand{\R}{\mathbb{R}}
\newcommand{\Z}{\mathbb{Z}}

\DeclareMathOperator{\Amp}{Amp}

\def\d{\mbox{d}}
\DeclareMathOperator{\Eff}{Eff}

\def\im{\mbox{Im}}

\def\Ker{\mbox{Ker}}

\def\NS{\mbox{NS}}

\DeclareMathOperator{\Pos}{Pos}
\DeclareMathOperator{\pr}{pr}
\def\rank{\mbox{rank}}

\def\Sing{\mbox{Sing}}

\def\Tors{\mbox{Tors}}

\DeclareMathOperator{\Pseff}{Pseff}
\DeclareMathOperator{\Nef}{Nef}

\DeclareMathOperator{\rk}{rk}

\DeclareMathOperator{\Id}{Id}
\DeclareMathOperator{\tr}{tr}
\DeclareMathOperator{\CI}{CI}
\DeclareMathOperator{\HR}{HR}

\numberwithin{equation}{section}

\begin{document}
\title[]{Generalized Bogomolov Inequalities}

\author{Mihai Pavel, Julius Ross, Matei Toma}

\address{Institute of Mathematics of the Romanian Academy,
P.O. Box 1-764, 014700 Bucharest, Romania
}
\email{cpavel@imar.ro}
\address{Department of Mathematics, Statistics, and Computer Science, University of Illinois at Chicago, 322 Science and Engineering Offices (M/C 249), 851 S. Morgan Street, Chicago, IL 60607
 }
\email{juliusro@uic.edu}
 
\address{ Universit\'e de Lorraine, CNRS, IECL, F-54000 Nancy, France
}

\email{Matei.Toma@univ-lorraine.fr}

\date{\today}
\keywords{semistable coherent sheaves, Hodge-Riemann relations, Bogomolov inequality}
\subjclass[2020]{14D20, 32G13, 32J27}

\begin{abstract} 
We introduce the notion of a Hodge-Riemann pair of cohomology classes that generalizes the classical Hodge-Riemann bilinear relations, and the notion of a Bogomolov pair of cohomology classes that generalizes the Bogomolov inequality for semistable sheaves.  We conjecture that every Hodge-Riemann pair is a Bogomolov pair, and prove various cases of this conjecture.  As an application we get new results concerning boundedness of semistable sheaves. 
\end{abstract}
\maketitle
\setcounter{tocdepth}{1}

\section{Introduction}

Suppose that $X$ is a compact complex manifold of dimension $d\ge 2$. 
When we work algebraically we will assume $X$ is a complex projective manifold; when we work analytically we will assume $X$ is a compact K\"ahler manifold.  In either case consider $$\eta_{d-1}\in H^{d-1,d-1}(X)\text{ and } \eta_{d-2}\in H^{d-2,d-2}(X).$$   Assume that each $\eta_{i}$ is ``positive'', which in the algebraic case we mean lying in the ample cone $\Amp^i(X)$, and in the analytic case in the interior $\cK^i(X)$ of the nef cone (see \cref{sec:setup} for the precise definitions of these cones).

\begin{definition}[Hodge-Riemann pairs of cohomology classes]\label{def:algebraicHR:intro}
We say $(\eta_{d-1},\eta_{d-2})$ is a \textit{Hodge-Riemann pair} if for any $\alpha$ in $ N^1(X)$ (resp. $H^{1,1}(X)$ in the analytic case)
$$\int_X \alpha \cdot \eta_{d-1} = 0 \Rightarrow \int_X \alpha^2 \cdot \eta_{d-2} \le 0$$ 
with equality if and only if $\alpha = 0$. (See also the more precise \cref{def:algebraicHRpair}.)
\end{definition}

The terminology comes from the fact that the classical Hodge-Riemann bilinear relations imply that if $h$ is the class of an ample divisor on $X$ then $(h^{d-1},h^{d-2})$ is a Hodge-Riemann pair.  This extends to the analytic case in which if $\omega$ is a K\"ahler form then $([\omega]^{d-1},[\omega]^{d-2})$ is a Hodge-Riemann pair.   

There are other natural Hodge-Riemann pairs that come from Schur polynomials $s_{\lambda}$.  To describe these, for any symmetric homogeneous polynomial $p$ in variables $x_1,\ldots,x_e$ we define the \emph{derived polynomials} $p^{(i)}$ by the rule
$$p(x_1+t,\ldots,x_{e}+t) = \sum_{i=0}^{\deg p}t^i p^{(i)}(x_1,\ldots,x_e).$$
Clearly $p^{(i)}$ is a symmetric homogeneous polynomial of degree $\deg p -i$, and $p^{(0)}=p$.  For simplicity we write $p' = p^{(1)}$.    If $A$ is a vector bundle we denote by $s_{\lambda}(A)$ the Schur class of $A$ and similarly for $s'_{\lambda}(A)$. 

\begin{proposition}[$\subset$ Proposition \ref{prop:ShurAmpleHRpair}] \label{prop:ShurAmpleHRpair:intro}
Let $X$ be a complex projective manifold of dimension $d\ge 2$,  let $\lambda$ be a partition of $d-1$ and assume $A$ is an ample vector bundle of rank $e\ge d-1$.    Then
$$(s_\lambda(A),s_\lambda'(A))$$ is a Hodge-Riemann pair. 

In particular, the pair $(c_{d-1}(A),c_{d-2}(A))$ of Chern classes as well as the pair $(s_{d-1}(A),s_{d-2}(A))$ of Segre classes are Hodge-Riemann pairs.
\end{proposition}

When $A=\oplus_{i=1}^e L_i$ is a direct sum of ample line bundles $L_i$ this gives Hodge-Riemann pairs whose elements are certain polynomials in $c_1(L_1),\ldots,c_1(L_e)$.  This extends analytically to K\"ahler classes:

\begin{proposition}[$\subset$ Proposition \ref{prop:ShurHRpair}]\label{prop:ShurHRpair:intro}
Let $\alpha_1,\ldots,\alpha_e$ be K\"ahler classes on a compact complex manifold $X$ of dimension $d\ge 2$. Suppose that $e \ge d-1$ and let $\lambda$ be a partition of $d-1$. Then $$(s_\lambda(\alpha_1,\ldots,\alpha_e),s_\lambda'(\alpha_1,\ldots,\alpha_e))$$ is a Hodge-Riemann pair.
\end{proposition}

 In fact our proof is stronger, and \cref{prop:ShurHRpair} gives a pointwise statement about analogous polynomials of K\"ahler forms.
\begin{center}
*
\end{center}

Our motivation for introducing Hodge-Riemann pairs is a belief that they form the right setting for a generalization of the Bogomolov inequality.  Any positive class $\eta_{d-1}\in H^{d-1,d-1}(X)$ defines for each torsion-free coherent sheaf $E$ on $X$ a slope 
$$\mu_{\eta_{d-1}}(E): = \frac{\int_X c_1(E) \cdot \eta_{d-1}}{\rank(E)}$$
from which one gets a notion of  (semi)stability with respect to $\eta_{d-1}$.

\begin{conjecture}
Suppose $(\eta_{d-1},\eta_{d-2})$ is a Hodge-Riemann pair.  Then for any $\eta_{d-1}$-semistable torsion-free sheaf $E$ of rank $r$ on $X$ we have
$$\int_X (2rc_2(E) - (r-1)c_1(E)^2)\cdot \eta_{d-2}\ge 0.$$
\end{conjecture}

When the conclusion of this conjecture holds we call $(\eta_{d-1},\eta_{d-2})$ a \emph{Bogomolov pair}.  The terminology comes from the fact that the classical Bogomolov inequality \cite[Theorem 7.3.1]{HL} states that if $h$ is an ample class then $(h^{d-1},h^{d-2})$ is a Bogomolov pair.

In this paper we prove various special cases of this conjecture. For example we know this conjecture holds for threefolds (\cref{prop:complete_intersections}) and for complex tori (\cref{cor:tori}). We also have the following:

\begin{theorem}[= Theorem \ref{thm:HRimpliesBG}]\label{thm:HRimpliesBG:intro}
Suppose $X$ is a compact complex manifold of dimension $d$.  Let $\alpha_1,\ldots,\alpha_e$ be K\"ahler classes on $X$ with $e\ge d-1$, and let $\lambda$ be a partition of length $d-1$.  Then 
$$(s_{\lambda}(\alpha_1,\ldots,\alpha_e),s'_{\lambda}(\alpha_1,\ldots,\alpha_e))$$
is a Bogomolov pair.
\end{theorem}

Our proof of this theorem is analytic; when $E$ is a stable vector bundle we use the Hitchin-Kobayashi correspondence and the same computation due to L\"ubke, \cite{Luebke}, that computes the pointwise discriminant with respect to the Hermitian-Einstein metric taken with respect to the Gauduchon metric $\sqrt[d-1]{s_{\lambda}(\alpha_1,\ldots,\alpha_e)}.$   We then extend this to apply to any stable torsion free sheaf using a resolution, and then to semistable sheaves using induction on the rank.   

The proofs of the remaining theorems are independent of this, and do not rely on the L\"ubke computation.

\begin{theorem}[= Corollary \ref{cor:segrebogomolovpairI}]\label{cor:segrebogomolovpair:intro}
Let $A$ be an ample vector bundle of rank at least $d-1$ on $X$ such that
\[
    s_d(A) \ge \mu_{\max, s_{d-1}(A)}(A),
\]
where $\mu_{\max,s_{d-1}(A)}(A)$ denotes the maximal slope with respect to $s_{d-1}(A)$ of non-trivial subsheaves of $A$. Then the Segre classes $$(s_{d-1}(A),s_{d-2}(A))$$ form a Bogomolov pair.
\end{theorem}

We expect the analogous statement to hold also for Schur classes, but can only prove this in a special case:

\begin{theorem}[= \cref{thm:globallygenerated_chern}]\label{thm:globallygeneratedintro_chern}
Suppose that $A$ is a globally generated and ample vector bundle of rank at least $d-1$ on $X$ such that
\[
    c_d(A) > \mu_{\max, c_{d-1}(A)}(A).
\]
Then
$$(c_{d-1}(A),c_{d-2}(A))$$
is a Bogomolov pair.
\end{theorem}

\begin{center}
*
\end{center}

As an application we get new results concerning boundedness of torsion-free semistable sheaves of a given topological type (see \cref{def:boundedFam} for the definition of a \textit{bounded} set). The following technical statement gives a general condition under which we can prove boundedness of such sheaves that are semistable with respect to a class $\eta_{d-1} \in \Amp^{d-1}(X)$ as long as there is an $\eta_{d-2}$ making $(\eta_{d-1},\eta_{d-2})$ a Hodge-Riemann and Bogomolov pair. In fact it proves more in that this pair can vary in a compact set:

\begin{theorem}[$\subset$ Theorem \ref{thm:MainBound}]
Let $X$ be a projective manifold of dimension $d$. Let $K \subset \Amp^{d-1}(X) \times \Amp^{d-2}(X)$ be a path-connected compact subset, and denote by $K' := \pr_1(K)$ and $K'' := \pr_2(K)$ its two corresponding projections. Suppose that
\begin{enumerate}
    \item there is an element $(h^{d-1},h^{d-2})\in K$ for some $h\in \Amp^1(X)$, and
    \item for every $\eta_{d-1} \in K'$ there exists $\eta_{d-2} \in K''$ and a path $$\gamma_{(\eta_{d-1},\eta_{d-2})} : [0,1] \to K$$ connecting $(\eta_{d-1},\eta_{d-2})$ to $(h^{d-1},h^{d-2})$ such that for all $t \in [0,1]$ the pair $\gamma_{(\eta_{d-1},\eta_{d-2})}(t)$ is a Hodge-Riemann and Bogomolov pair. 
\end{enumerate}
Then the set of isomorphism classes of torsion-free sheaves of fixed rank $r$ and fixed Chern classes $c_i \in N^i(X)$  that are $\eta_{d-1}$-semistable with respect to some $\eta_{d-1}\in K'$ is bounded.
\end{theorem}

Theorem \ref{thm:MainBound} is actually more general in that it also holds for compact K\"ahler manifolds, and $(h^{d-1},h^{d-2})$ can be 
replaced by any given pair for which  a suitable boundedness is already known (see Definition \ref{def:boundednessclass}).

From \cite{MegyPavelToma} this boundedness result is enough to ensure that each $\eta_{d-1} \in K'$ defines a finite type moduli space of $\eta_{d-1}$-semistable sheaves, and $K'$ has a chamber structure separated by walls that determine when these moduli spaces change.   Using the results already stated, this boundedness applies in the following case:

\begin{corollary}[= Corollary \ref{cor:boundednessschurkahler}]
Let $K'$ be a path-connected  and compact set of K\"ahler classes on $X$ that includes a rational point, $\lambda$ is a partition of length $d-1$ and $e\ge d-1$. Then the set of isomorphism classes of torsion-free sheaves of given topological type that are semistable with respect to some element of 
$$  \{ s_{\lambda}(\alpha_1,\ldots,\alpha_e) \, \mid \, \alpha_1,\ldots,\alpha_e\in K'\}$$
is bounded.
\end{corollary}

\subsection*{Acknowledgments: } We thank Ziquan Zhuang who kindly pointed out an error in the argument of our first version of \cref{prop:ss-projective-bundle}.  

MP was partially supported by the PNRR grant CF 44/14.11.2022 \textit{Cohomological Hall algebras of smooth surfaces and applications}, and by a grant of the Ministry of Research, Innovation and Digitalization, CNCS-UEFISCDI, project number PN-IV-P2-2.1-TE-2023-2040, within PNCDI IV.  JR is supported by Simons Foundation Award.  MT acknowledges financial support from  IRN ECO-Maths.

\section{Set-up and notation}\label{sec:setup}
We will be concerned in this paper with generalizing the Bogomolov inequality for torsion-free semistable sheaves in two related contexts, over smooth complex projective varieties and over compact K\"ahler manifolds. 

\subsection{Algebraic set-up} We consider here complex polarized smooth projective varieties $(X,h)$ of dimension $d$, where $h$ is an integral ample class on $X$. We denote by  $N^p(X)$ the numerical group of real codimension $p$ cycles on $X$, by  $\overline{\Eff}^p(X)$ the closed convex cone generated by effective $p$-codimensional cycles, by $\Nef^{d-p}(X)$ its dual cone in $N^{d-p}(X)$,   and by $\Amp^{p}(X)$ the interior of $\Nef^p(X)$.

 The cones $\Nef^p(X)$ are full-dimensional for $0\le p\le d$ (also for singular $X$ by \cite[Lemma 3.7]{FulgerLehmann2017cones}), so their interiors $\Amp^{p}(X)$ are nonempty. 
 In degree $1$, one recovers the cone of real ample divisor classes, $\Amp(X)=\Amp^1(X)$, \cite{Kleiman}, whereas in degree $d-1$ by \cite{bdpp} $\Nef^{d-1}(X)$ is the movable cone and $\Amp^{d-1}(X)$ is the cone of mobile curve classes on $X$, cf. \cite[Definition 11.4.16]{Lazbook2}. 
We will call the elements of $\Amp^{p}(X)$ {\em ample $p$-classes.} 
Note that there is a natural non-degenerate pairing $N^p(X)\times N^{d-p}(X)\to\R$.

\subsection{Analytic set-up} 

Here we will work with polarized compact K\"ahler manifolds $(X,[\omega])$, where $[\omega]\in \cK^1(X)$ is a fixed K\"ahler class associated to a K\"ahler form $\omega$ on $X$. Here $\cK^1(X)$ is the cone of K\"ahler classes on $X$. We now describe two types of positive cones to be considered inside the subspaces $H^{p,p}(X):=H^{p,p}_{dR}(X)_\R$ of de Rham cohomology classes represented by real closed $(p,p)$-forms (or alternatively by real closed $(p,p)$-currents) on $X$. We recall that for compact K\"ahler manifolds the canonical maps $H^{p,p}_{BC}(X)\to H^{p,p}_{dR}(X)\to H^{p,p}_A(X)$ between Bott-Chern, de Rham and Aeppli cohomology groups are isomorphisms. We will occasionally identify these cohomology groups in this canonical way without further comment. 

For differential forms we use the terminology of \cite{HarveyKnapp} for (weak, regular, and strong) positivity of forms, with strict positivity for a $(p,p)$-form meaning that it belongs to the interior of the corresponding cone of positive forms. 
In particular, a real $(p,p)$-form $\eta$ on $X$ is strictly weakly positive if and only if its restriction to any immersed $p$-dimensional submanifold is a volume form, cf. \cite[p.~46]{HarveyKnapp}. In this paper only weak and strong positivity for forms or currents will be used. The corresponding order relations will be indicated by $\le_w$ and $\le_s$.  Note that these positivity notions coincide in degrees $0,1,d-1$ and $d$.

We define $\Pseff^p(X)\subset H^{p,p}(X)$ to be the convex cone generated by classes of {\em strongly} positive closed $(p,p)$-currents on $X$
and $\Nef^p_A(X)$ to be the pull-back by $H^{p,p}_{dR}(X)\to H^{p,p}_A(X)$ of the cone
$$\{ a\in  H^{p,p}_A(X) \ \mid \ \forall \varepsilon>0 \ \exists \alpha_\varepsilon\in a \ \alpha_\varepsilon\ge_w-\varepsilon\omega^p\}\subset H^{p,p}_A(X). $$
(This might differ for $2\le p \le d-2$ from other work in which weak positivity of currents is considered.) 

Standard techniques in complex geometry can be used to prove the following (details are in the Appendix):

\begin{proposition}\label{prop:cones}
    For a compact K\"ahler manifold $X$ of dimension $d$ and $0\le p\le d$ the convex cones $\Pseff^p(X)$ and $\Nef^{d-p}_A(X)$ are closed, salient and dual to each other with respect to the usual intersection form.  
\end{proposition}

We denote by $\cK^p(X)$ the interior of $\Nef^p_A(X)$. This notation agrees with the  fact that the interior of $\Nef^1_A(X)$ is the cone of K\"ahler classes on X, cf. \cite[Section 1]{DemaillyPaun}, \cite[Proposition 2.7]{ChioseRasdeaconuSuvaina2019}.

\begin{remark}\label{rem:cones}
    When $X$ is moreover projective, we may look at the following real vector subspaces of $H^{p,p}(X)$:
    $$\cC^p(X)\subset \NS^p(X)_\R:=(\im(H^{2p}(X,\Z)\to H^{2p}(X,\R))\cap H^{p,p}(X))\otimes\R,$$
    where $\cC^p(X)$ denotes the subspace spanned by classes of $p$-codimensional algebraic cycles on $X$. The inclusions $\cC^p(X)\subset \NS^p(X)_\R$ are equalities for $p\in\{0,1,d-1,d\}$ and there are obvious projections $\pi_p:\cC^p(X)\to N^p(X)$. These projections are known to be isomorphisms for $p\in\{0,1,2,d-1,d\}$ and we will identify $\cC^p(X)$ with $N^p(X)$ in these cases.  For  $p=1$ and $p=d-1$ one has $\overline{\Eff}^p(X)=\Pseff^p(X)\cap\cC^p(X)$ and $\Nef^p(X)=\Nef^p_A(X)\cap \cC^p(X)$, \cite[Proposition 6.1]{Demailly-regularization}, \cite{Witt}. We note also that for these values of $p$ the elements of $\cK^p(X)$ are represented by strictly positive closed $(p,p)$-forms. In general one only has $\overline{\Eff}^p(X)\subset\pi_p(\Pseff^p(X)\cap\cC^p(X))$ and $\Nef^p(X)\supset\pi_p(\Nef^p_A(X)\cap\cC^p(X)).$ In fact these inclusions may be strict even for $p=2$ and $d=4$ as shown in \cite[Theorem B]{DELV11}. 
\end{remark}
\subsection{$\mathbb R$-twisted vector bundles}

Let $X$ be a smooth projective $d$-dimensional variety, $A$ be an ample vector bundle of rank $e$ on $X$, $\pi : \P(A) \to X$ be the natural projection, and $\xi = c_1(\cO_{\P(A)}(1))$ be the Chern class of the tautological line bundle on $\P(A)$.   

We use the notation of $\mathbb R$-twisted bundles (see \cite[Sec. 2.4]{RossToma1} for a longer account).  Let $h$ be an ample class on $X$.   For $t\in \mathbb R$, the notation $A\langle th\rangle$ is a formal object whose Chern classes are defined by the rule
$$c_p(A\langle th\rangle) = \sum_{k=0}^p \binom{e-k}{p-k} c_k(A) (th)^{p-k}.$$
This definition is made so that when $t\in \mathbb Z$ we have $c_p(A\langle th\rangle) = c_p(A\otimes \mathcal O(th))$.   As a space, the projectivization of $A\langle th\rangle$ is just $\mathbb P(A)$ but the tautological class $\xi$ is replaced by
$$\xi_t := \xi + t \pi^* h.$$
We say that $A\langle th\rangle$ is ample if $\xi_t$ is ample.

\subsection{Further Notation and Conventions}
Given a symmetric homogeneous polynomial $p$ in $e$ variables and a coherent sheaf $E$ on a complex manifold $X$ the class $p(E)\in H^{\deg p,\deg p}(X)$ is defined by writing $p$ as a polynomial in the elementary symmetric polynomials and taking the corresponding polynomial in the Chern classes of $E$.  When $E$ is a vector bundle one may equivalently define $p(E)=p(\alpha_1,\ldots,\alpha_e)$ where the $\alpha_i$ are the Chern roots of $E$. If $p$ is the complete homogeneous symmetric polynomial of degree $i$, one obtains the Segre class $s_i(E)$ (cf. the convention in \cite[p.~28]{Fulton-Pragacz}). The Segre and Chern classes are related by $s(E) = c(E^\vee)^{-1}$.

Given a hermitian metric $h$ on a complex vector bundle on $X$ we let $F_h$ denote the curvature of the Chern connection.   This gives rise to Chern forms $c_i(E,h)$ whose class $[c_i(E,h)]$ in Bott-Chern cohomology is independent of the choice of $h$ \cite[Chapter 2.4]{Bismut_superconnections}.

\section{Hodge-Riemann pairs}\label{sec:HRpairs}

\begin{definition}\label{def:HRproperty} 
Let $X$ be a complex smooth projective variety (resp.\ a compact K\"ahler manifold) of dimension $d$.  Let $h\in \Amp^1(X)$ and $\eta_{d-2}\in \Amp^{d-2}(X)$ (resp. $h\in \cK^1(X)$ and $\eta_{d-2}\in \cK^{d-2}(X)$).  

We say that $\eta_{d-2}$ has the \emph{Hodge-Riemann property} with respect to $h$ if  
for all $\alpha$ in $N^1(X)$ (resp.\ $H^{1,1}(X)$) we have
\begin{equation}\int_X \alpha \cdot \eta_{d-2} \cdot h =0 \Rightarrow \int_X \alpha^2 \cdot \eta_{d-2}\le 0\label{eq:HR}\end{equation}
with equality if and only if $\alpha=0$.
\end{definition}

Equivalently $\eta_{d-2}$ has the Hodge-Riemann property if and only if the intersection matrix
$$ Q_{\eta_{d-2}}(\alpha,\alpha') := \int_X \alpha \cdot \eta_{d-2} \cdot \alpha' \text{ for } \alpha,\alpha \text{ in } N^1(X)\text{ (resp.\ } H^{1,1}(X)\text{)}$$
has signature $(+,-,\cdots,-)$.  One can check that if $\eta_{d-2}$ has the Hodge-Riemann property with respect to some $h$ in $\Amp^1(X)$ (resp.\ in $\cK^1(X)$) then it has the Hodge-Riemann property with respect to any $h$ in $\Amp^1(X)$ (resp.\ in $\cK^1(X)$).

Note that if $\eta_{d-2}$ has the Hodge-Riemann property then the map $N^1(X)\to N^{d-1}(X)$  (resp.\  $H^{1,1}(X)\to H^{d-1,d-1}(X))$ given by
$$\alpha\mapsto \alpha\cdot \eta_{d-2}$$ is an isomorphism.  When this occurs, given $\gamma$ in $N^{d-1}(X)$ (resp.\ in  $H^{d-1,d-1}(X)$) we define $\gamma/\eta_{d-2}$ in $N^1(X)$ (resp.\ in  $H^{1,1}(X)$) by requiring
$$(\gamma/\eta_{d-2}) \cdot \eta_{d-2} = \gamma.$$

The classical Hodge-Riemann bilinear relations imply that if $[\omega]$ is a K\"ahler class on $X$ then $[\omega]^{d-2}$ has the Hodge-Riemann property (so this implies $h^{d-2}$ has the Hodge-Riemann property if $h\in \Amp^1(X)$).  In \cite{DinhNguyen06} it is shown that products $[\omega_1] \cdots[\omega_{d-2}]$ of K\"ahler classes $[\omega_i]$ also have the Hodge-Riemann property. 

In \cite{RossToma1} it is shown that Schur classes of ample vector bundles give rise to classes with the Hodge-Riemann property.  Precisely, if $A$ is an ample vector bundle of sufficiently high rank, and $\lambda$ is a partition of length $d-2$ then $s_{\lambda}(A)$ has the Hodge-Riemann property.  In \cite{RossToma3} it is shown that if $\lambda$ is a partition of length $d-2$ and $\alpha_1,\ldots,\alpha_e$ are (possibly irrational) ample classes then $s_{\lambda}(\alpha_1,\ldots,\alpha_e)$ has the Hodge-Riemann property.

\begin{definition}[Hodge-Riemann pairs of cohomology classes]\label{def:algebraicHRpair}

Let $X$ be a complex smooth projective variety (resp.\ a compact K\"ahler manifold) of dimension $d$.  Let $\eta_{d-2}\in \Amp^{d-2}(X)$ (resp.  $\eta_{d-2}\in \cK^{d-2}(X)$).  Suppose also $\eta_{d-1}\in N^{d-1}(X)$ (resp.\ $\eta_{d-1}\in H^{d-1,d-1}(X)$).

We say $(\eta_{d-1},\eta_{d-2})$ is a \textit{Hodge-Riemann pair} if the following holds:

\begin{enumerate}
\item $\eta_{d-2}$ has the Hodge-Riemann property with respect to some $h$ in $\Amp^1(X)$ (resp.\ in $\cK^1(X)$).
\item $\int_X h \cdot \eta_{d-1}>0$
\item We have
$$\int_X \eta_{d-2}\cdot (\eta_{d-1}/\eta_{d-2})^2>0.$$
\end{enumerate}
\end{definition}

\begin{definition}
For any $\eta_{d-2}\in \Amp^{d-2}(X)$ (resp.  $\eta_{d-2}\in \cK^{d-2}(X)$) having the Hodge-Riemann property with respect to some $h$ in $\Amp^1(X)$ (resp.\ in $\cK^1(X)$), we write 
$\Pos_{\eta_{d-2}}(X)$ for the set of elements $\beta\in N^{1}(X)$ (resp. $\beta\in H^{1,1}(X)$) such that $\int_X \beta \cdot \eta_{d-2} \cdot h >0$ and $ \int_X \beta^2 \cdot \eta_{d-2}>0$.
\end{definition}

Clearly $\Pos_{\eta_{d-2}}(X)$ is a quadratic cone in $N^{1}(X)$ (resp in $ H^{1,1}(X)$). Then with this notation 
$$(\eta_{d-1},\eta_{d-2})\text{ is a Hodge-Riemann pair} \Longleftrightarrow \eta_{d-1}\in\eta_{d-2}\Pos_{\eta_{d-2}}(X).$$

In most of our applications, $\eta_{d-1}$ will itself have some positivity, in which case the above definition agrees with that given in the introduction:

\begin{lemma}\label{lem:equivalentHRpairforpositive}
With notation as in Definition \ref{def:algebraicHRpair} suppose in addition that there is an $h\in \Amp^1(X)$ (resp.\ in $\cK^1(X)$) such that  $$\int_X h\cdot \eta_{d-1}>0.$$
Then $(\eta_{d-1},\eta_{d-2})$ is a Hodge-Riemann pair if and only if for all $\alpha$ in $N^1(X)$ (resp.\ $H^{1,1}(X)$)
$$\int_X \alpha \cdot \eta_{d-1} =0 \Rightarrow \int_X \alpha^2 \cdot \eta_{d-2}\le 0$$
with equality if and only if $\alpha=0$.
\end{lemma}
\begin{proof}
Suppose first that the conclusion of this statement holds. The space of those $\alpha$ such that $\int_X \alpha\cdot \eta_{d-1}=0$ has codimension $1$.  Thus $Q_{\eta_{d-2}}$ has one strictly positive eigenvalue given by $h$ and all other eigenvalues strictly negative.  Thus $\eta_{d-2}$ has the Hodge-Riemann property with respect to $h$.

Set $\beta = \eta_{d-1}/\eta_{d-2}$.  Then there is a unique $t\in \mathbb R$ so that $(\beta - th)\eta_{d-2}\beta=0$.  By the Hodge-Riemann property $\int_X (\beta-th)^2\cdot \eta_{d-2}\le 0$ which rearranges to give $\int_X \beta^2 \cdot \eta_{d-2} >0$.

In the other direction suppose that $(\eta_{d-1},\eta_{d-2})$ is a Hodge-Riemann pair and $\int_X \alpha\cdot \eta_{d-1}=0$.  Again set $\beta = \eta_{d-1}/\eta_{d-2}$.  Then
$$\int_X \eta_{d-2}\cdot \alpha^2 \int_X \eta_{d-2} \cdot \beta^2 \le \left(\int_X \eta_{d-2}\cdot \alpha\cdot \beta \right)^2= \left(\int_X \eta_{d-1}\cdot \alpha\right)^2 =0$$
with equality if and only if $\alpha$ is proportional to $\beta$.  Since by hypothesis $\int_X \eta_{d-2}\cdot \beta^2>0$ this implies $\int_X\alpha^2 \cdot \eta_{d-2}\le 0$ with equality if and only if $\alpha=0$ which completes the proof.
\end{proof}

\begin{lemma}\label{lem:HRopen}
    The set of Hodge-Riemann pairs is open in $N^{d-1}(X)\times\Amp^{d-2}(X)$ (resp. in $H^{d-1,d-1}(X)\times\cK^{d-2}(X)$).
\end{lemma}
\begin{proof}
    The set $\HR^{d-2}(X)$ of Hodge-Riemann classes is open in $\Amp^{d-2}(X)$ (resp. in $\cK^{d-2}(X)$) and the set of Hodge-Riemann pairs of $X$ may be seen as a subbundle in open cones over $\HR^{d-2}(X)$ inside the trivial real vector bundle $N^{d-1}(X)\times \HR^{d-2}(X)$  (resp. in $H^{d-1,d-1}(X)\times \HR(X)$) over $\HR(X)$.
\end{proof}

\begin{remark}\label{rem:DELV}
Suppose that  $X$ is a projective manifold and take $$(\eta_{d-1},\eta_{d-2})\in(\cK^{d-1}(X)\cap\cC^{d-1}(X))\times (\cK^{d-2}(X)\cap\cC^{d-2}(X)).$$ If the pair $(\eta_{d-1},\eta_{d-2})$ has the Hodge-Riemann property, then the corresponding pair of numerical classes in $\Amp^{d-1}(X)\times\Amp^{d-2}(X) $ has the Hodge-Riemann property too. But the converse does not hold in general, as the following example shows.
\end{remark}
\begin{example}\label{ex:delv}
Let $X=A\times A$ be a self-product of a very general principally polarized abelian surface $(A,\theta)$. In \cite{DELV11} the authors describe the numerical cohomology ring $N^*(X)$ of $X$ and various positive cones inside $N^*(X)$. We use their notation 
and consider the following classes:  $\theta_1=\pr_1^*(\theta)$, $\theta_2=\pr_2^*(\theta)$, $\lambda=c_1(\cP)$, where $\cP$ is the Poincar\'e bundle on $A\times A$. These form a basis of  $N^1(X)$ and generate the numerical cohomology ring $N^*(X)$. 
Putting 
$A=U/\Lambda$, $V=U\times U$, $X=V/(\Lambda\times\Lambda)$ one may choose coordinates $(z_1,z_2,z_3,z_4)$ on $V$ such that this basis gets represented as $\theta_1=i\d z_1\wedge \d\bar z_1 +i\d z_2\wedge \d\bar z_2$, $\theta_2=i\d z_3\wedge \d\bar z_3 +i\d z_4\wedge \d\bar z_4$, $\lambda=i\d z_1\wedge \d\bar z_3 +i\d z_2\wedge \d\bar z_4$. Consider now the class $\eta=\theta_1\theta_2$. The computations in the proof of \cite[Proposition 4.4]{DELV11} show that $\eta$ restricts positively to any $2$-dimensional complex subspace of $V$, so $\eta$ belongs to $\cK^2(X)$, hence also to $\Amp^2(X)$. The matrix of the intersection form that $\eta$ defines on $N^1(X)$ with respect to the above basis is
$$\left(\begin{array}{ccc} 0 & 4 &0\\
4 & 0 &0\\
0 & 0&-4\\
\end{array}\right),$$
showing that $\eta$ has the Hodge-Riemann property with respect to $N^1(X)$. However 
$\eta\wedge(i\d z_1\wedge \d\bar z_2 +i\d z_2\wedge \d\bar z_1)=0$ and thus the intersection form defined by $\eta$ on $H^{1,1}(X)$ has zero eigenvalues. 
Take now $h=\theta_1+\theta_2$. 
Then we get $h\in\Amp^1(X)$, $\eta h=\frac{1}{3}h^3\in\Amp^3(X)$ and $(\eta h,\eta)$ is a Hodge-Riemann pair (with respect to $N^1(X)$).
\end{example}

\begin{proposition}\label{prop:ShurAmpleHRpair}
Let $X$ be a complex projective manifold of dimension $d\ge 2$, and let $\lambda$ be a partition of $d-1$. Also let $h\in \Amp^1(X)$ and $t\in \mathbb R_{\ge 0}$.  

If $A\langle th\rangle$ is an ample vector bundle of rank $e \ge d-1$, then 
$$(s_\lambda(A\langle th\rangle),s_\lambda'(A\langle th\rangle))$$ is a Hodge-Riemann pair.
\end{proposition}
\begin{proof}
The proof is essentially the one of \cite[Theorem 10.2]{RossToma2}.  Set $A':= A\langle th\rangle$. First we observe that $s_{\lambda}(A')\in \Amp^{d-1}(X)$ by \cite{FL}, and also $s_{\lambda}'(A')\in \Amp^{d-2}(X)$ since $s_{\lambda}'$ is Schur positive. 

Suppose that $\alpha\in H^{1,1}(X)$ satisfies $\int_X \alpha \cdot s_{\lambda}(A')=0$.     Consider $\hat{X} = X\times \mathbb P^1$ and let $\tau$ denote the hyperplane class on $\mathbb P^1$ and let $\hat{A} = A'\boxtimes \mathcal O_{\mathbb P^1}(1)$ which is an ample bundle on $\hat{X}$ and $s_{\lambda}(\hat{A}) = s_\lambda(A') + s'_\lambda(A')\tau$.   Also set $\hat{h} = h + \tau$ which is ample on $\hat{X}$.

By \cite[Theorem 5.3]{RossToma1} $s_{\lambda}(\hat{A'}) \in H^{d-1,d-1}(\hat{X})$ has the Hodge-Riemann property with respect to $\hat{h}$, so if $\hat{\alpha} \in H^{1,1}(\hat{X})$ satisfies $\int_{\hat{X}} \hat{\alpha} \cdot s_\lambda(\hat{A})\cdot \hat{h}=0$ then $\int_{\hat{X}} \hat{\alpha}^2\cdot s_{\lambda}(\hat{A})\le 0$ with equality if and only if $\hat{\alpha}=0$.

Now we apply this to
$$\hat{\alpha} = \alpha - \frac{\int_X \alpha \cdot s_\lambda'(A')\cdot h}{\int_X s_\lambda(A')\cdot h} \tau$$
(observing that $\int_X s_\lambda(A')\cdot h>0$ as $A'$ is ample \cite{FultonLazarsfeld}).  Using our assumption that $\int_X \alpha \cdot s_{\lambda}(A')=0$ one checks that $\int_{\hat{X}} \hat{\alpha}\cdot  s_{\lambda}(\hat{A})\cdot\hat{h}=0$ and we are done.
\end{proof}

We now turn to the analogous pointwise definitions.

\begin{definition}[Pointwise Hodge-Riemann Property]\label{def:analyticHR}
Let $X$ be a complex manifold  and let $\Omega_{d-2}\in \Omega^{d-2,d-2}(X)$ be $\partial \bar{\partial}$-closed and strictly weakly positive. We say that
$\Omega_{d-2}$ has the Hodge-Riemann property with respect to a strictly  positive $(1,1)$-form $\omega$ if for all points $x\in X$ we have
$$\Omega_{d-2}(x) \wedge \omega(x)^2>0$$
and for all $\alpha \in \Omega^{1,1}(X)$ we have
$$\Omega_{d-2}(x) \wedge \omega(x) \wedge \alpha(x)=0 \Rightarrow \alpha(x)^2 \wedge \Omega_{d-2}(x)\le 0$$
with equality if and only if $\alpha(x)=0$.
\end{definition}

Just as for cohomology classes, when $\Omega_{d-2}$ has the Hodge-Riemann property it defines for each $x\in X$ a bilinear form on $\Lambda^{1,1}T^*_x$ with signature $(+,-,\ldots,-)$, and so the map $\tau \mapsto \tau \wedge \Omega_{d-2}(x)$ from $\Lambda^{1,1}T_x^*\to \Lambda^{d-1,d-1} T_x^*$ is an isomorphism.   When this occurs for $\gamma \in \Lambda^{d-1,d-1} T_x^*$ we define $(\gamma/\Omega_{d-2}(x))\in \Lambda^{1,1}T_x^*$ by requiring
$$(\gamma/\Omega_{d-2}(x)) \wedge \Omega_{d-2}(x) = \gamma.$$

\begin{definition}[Hodge-Riemann pairs of differential forms]\label{def:analyticHRpair}
    Let $X$ be a compact complex manifold of dimension $d$ and let $\Omega_{d-1}$, $\Omega_{d-2}$ be $\partial \bar{\partial}$-closed forms of type $(d-1,d-1)$ and $(d-2,d-2)$ respectively.   Assume also that $\Omega_{d-2}$ is strictly weakly positive.

    We say $(\Omega_{d-1},\Omega_{d-2})$ is a \textit{Hodge-Riemann pair} if at each point of $x\in X$ the following holds

    \begin{enumerate}
\item    $\Omega_{d-2}$ has the pointwise Hodge-Riemann property with respect to some  strictly  positive form $\omega$ on $X$
\item $\Omega_{d-1}\wedge \omega >0$
\item $\Omega_{d-2} \wedge (\Omega_{d-1}/\Omega_{d-2})^2>0$.
\end{enumerate}
\end{definition}

\begin{remark}
A similar definition can be found in \cite[Section 2]{ChenWentworth}.
\end{remark}

Just as in the case of cohomology classes, this definition simplifies when $\Omega_{d-1}$ is also assumed to be positive.

\begin{lemma}\label{lem:analyticequivalentHRpairforpositive}
With the notation as in Definition \ref{def:analyticHRpair} suppose in addition that there exists a strictly positive $(1,1)$-form $\omega$ such that for all points $x\in X$ we have
$\Omega_{d-1}(x) \wedge \omega(x)>0$. 
 Then $(\Omega_{d-1},\Omega_{d-2})$ is a Hodge-Riemann pair if and only if for every $\alpha\in \Omega^{1,1}(X)$ it holds that
 $$\alpha \wedge \Omega_{d-1} = 0 \Rightarrow \alpha^2 \wedge \Omega_{d-2} \le 0$$ 
    at each point of $x \in X$, with equality at $x$ if and only if $\alpha(x) = 0$.
\end{lemma}
\begin{proof}
The proof is essentially the same that of Lemma \ref{lem:equivalentHRpairforpositive} and so omitted.
\end{proof}

\begin{remark}\label{rem:diffformpairisacohomologypair}
Assume $X$ is compact and K\"ahler.  If $(\Omega_{d-1},\Omega_{d-2})$ is a Hodge-Riemann pair then the cohomology classes $([\Omega_{d-1}],[\Omega_{d-2}])$ form a Hodge-Riemann pair (this follows easily from \cite[Corollary 5.4]{RossToma3}).    The converse does not hold, namely it is not the case that every Hodge-Riemann pair of cohomology classes can be represented as the classes of a pointwise Hodge-Riemann pair of differential forms.  In fact a positive class may not be representable by a positive form \cite{DELV11}.
\end{remark}

In \cite{RossToma3} we essentially prove that Schur polynomials of K\"ahler classes give rise to Hodge-Riemann pairs.  In fact this holds pointwise:

\begin{proposition}\label{prop:ShurHRpair}
Let $\omega_1,\ldots,\omega_e$ be K\"ahler forms on a complex manifold $X$ of dimension $d$. Suppose that $e \ge d-1$ and let $\lambda$ be a partition of $d-1$. Then $(s_\lambda(\omega_1,\ldots,\omega_e),s_\lambda'(\omega_1,\ldots,\omega_e))$ is a Hodge-Riemann pair.
\end{proposition}
\begin{proof}
The proof is similar to that of Proposition \ref{prop:ShurAmpleHRpair}.  Since this is a pointwise statement, we prove it in the linear case. So let $E$ and $F$ be two complex vector spaces of dimensions $d$ and $1$, respectively. We consider strictly positive $(1,1)$-forms $\omega_1,\ldots,\omega_e$ on $E$ and $\theta$ on $F$, and set 
\[
    \hat{s}_\lambda := s_\lambda(\omega_1 + \theta,\ldots,\omega_e + \theta) \in \bigwedge_{\R}^{d-1,d-1} (E \oplus F)^*.
\]
We have
\[
    \hat{s}_\lambda = s_\lambda(\omega_1,\ldots,\omega_e) + s'_\lambda(\omega_1,\ldots,\omega_e) \wedge \theta.
\]
We set 
\[
    \Omega := s_\lambda(\omega_1,\ldots,\omega_e), \quad \Omega' := s'_\lambda(\omega_1,\ldots,\omega_e), \quad \kappa := \frac{\alpha \wedge \Omega' \wedge \omega}{\Omega \wedge \omega}, \quad \hat{\alpha}:=\alpha-\kappa\theta,
\]
where $\omega$ is a fixed strictly positive $(1,1)$-form on $E$, and $\alpha$ is an arbitrary real $(1,1)$-form on $E$. 
Then 
\[
    \hat{s}_\lambda \wedge \omega^2 = \Omega' \wedge \omega^2>0
\]
and 
\[
    \hat{\alpha}\wedge\hat{s}_\lambda \wedge \omega =\alpha\wedge(\Omega+ \Omega'\wedge\theta) \wedge \omega-\kappa\theta\wedge(\Omega+ \Omega'\wedge\theta) \wedge \omega =(\alpha\wedge \Omega'\wedge\omega-\kappa\Omega\wedge\omega)\wedge\theta= 0.
\]
By the Hodge-Riemann property of Schur classes in the linear case for $\hat{s}_\lambda$ \cite[Theorem 10.2]{RossToma3} we get
\[
    \hat{\alpha}^2\wedge\hat{s}_\lambda \le 0
\]
with equality if and only if $\hat{\alpha}=0$, which is again equivalent to $\alpha=0$.

Now 
\[
    \hat{\alpha}^2\wedge\hat{s}_\lambda =
    (\alpha^2\wedge\Omega'-2\kappa\alpha\wedge\Omega)\wedge\theta
\]
and thus
\[
    (\alpha^2\wedge\Omega')(\Omega\wedge\omega)\le2(\alpha\wedge\Omega)(\alpha\wedge\Omega'\wedge\omega),
\]
from which the Hodge-Riemann property of the pair
$(\Omega,\Omega')$ directly follows.
\end{proof}

\begin{remark}\label{rmk:otherpolynomials}
One may extend the above results to derived Schur polynomials, or even products of derived Schur polynomials (proofs left to the reader; see \cite[Section 5.2]{RossToma1}).   It is not the case that this extends to every positive linear combination of Schur polynomials (see \cite[Remark 9.3]{RossToma1} for a related example), but it will hold for some linear combinations (see \cite[Section 9]{RossToma2}).
\end{remark}

\section{Bogomolov Pairs}\label{sec:bogomolovpairs}

In the following $X$ is either a complex smooth projective variety of dimension $d$ or a compact K\"ahler manifold of dimension $d$.  In the first case we let $\eta_{d-1}\in \Amp^{d-1}(X)$ and in the second we let $\eta_{d-1}\in \cK^{d-1}(X)$.

\begin{definition}[Slope-Semistability]\label{def:sstability}
If $E$ is a torsion-free coherent sheaf on $X$ we write
$$\mu_{\eta_{d-1}}(E) : = \frac{\int_X c_1(E) \cdot \eta_{d-1}}{\rk(E)}.$$

 A torsion-free coherent sheaf $E$ on $X$ is said to be \emph{semistable with respect to $\eta_{d-1}$} if for all proper coherent subsheaves $F\subset E$ we have
$$\mu_{\eta_{d-1}}(F)\le \mu_{\eta_{d-1}}(E).$$
We say $E$ is stable with respect to $\eta_{d-1}$ if strict inequality always holds.  We say it is polystable if $E=\oplus_i E_i$ with each $E_i$ stable and $\mu_{\eta_{d-1}}(E_i)=\mu_{\eta_{d-1}}(E)$.
\end{definition}

\begin{remark}
It is enough to check the slope inequalities only for saturated subsheaves $F \subset E$, cf. \cite[Corollary 2.14]{GKP-Movable}.  Thus $E$ is $\eta_{d-1}$-semistable if and only if for all torsion-free quotients $E\to G$ it holds that $\mu_{\eta_{d-1}}(G)\ge \mu_{\eta_{d-1}}(E)$ (with strict inequality needed for stability).
\end{remark}

\begin{lemma}\label{lem:stabilityisopen}
Suppose that a torsion-free sheaf $E$ is $\eta_{d-1}$-stable.  Then $E$ is $\eta$-stable for $\eta\in \Amp^{d-1}(X)$ (resp. $\eta\in \cK^{d-1}(X)$) sufficiently close to $\eta_{d-1}$.
\end{lemma}
\begin{proof}
We prove the statement when $X$ is projective. The K\"ahler case is similar and was shown in \cite[Corollary 6.9]{TomaLimitareaII}.

Set $C:= \mu_{\eta_{d-1}}(E)+1$.  For $\eta$ in some small ball $B$ around $\eta_{d-1}$ in $\Amp^{d-1}(X)$ we have $\mu_\eta(E)< C$.  The main ingredient of the proof is the fact that the set $\mathcal S$ of torsion-free quotient sheaves $E\to G$ such that $\mu_\eta(G)\le C$ for some $\eta \in B$ is bounded \cite[Theorem 3.1]{MegyPavelToma}.  

It is sufficient to check stability with respect to torsion-free quotients $E\to G$. For such a $G$ we have $\mu_{\eta_{d-1}}(G) > \mu_{\eta_{d-1}}(E)$ by the stability hypothesis on $E$.  So the boundedness of $\mathcal S$  implies that shrinking $B$ if necessary we can arrange that for all quotients in $\mathcal S$ we have $\mu_{\eta}(G) > \mu_{\eta}(E)$.  On the other hand, for the remaining torsion-free quotients $G$ we have $\mu_{\eta}(G)> C \ge \mu_{\eta}(E)$ and we are done.
\end{proof}

\begin{definition}[Discriminant]
Let $E$ be a torsion-free sheaf on $X$.  The \emph{discriminant} of $E$ is
$$\Delta(E) := 2\rk(E) c_2(E)  - (\rk(E)-1) c_1(E)^2.$$
\end{definition}

\begin{definition}[Bogomolov pairs]\label{def:bogomolovpairsofclasses}
We say that a pair $(\eta_{d-1},\eta_{d-2})$ is a \emph{Bogomolov pair} if
$$ E \text{ is semistable with respect to } \eta_{d-1} \Rightarrow \int_X \Delta(E) \cdot \eta_{d-2}  \ge 0.$$
\end{definition}

\begin{remark}
\begin{enumerate}
\item  If we want to consider the weaker condition that this inequality holds for a certain subclass of semistable sheaves $E$ we use the corresponding qualified definition.  For example, we will say that $(\eta_{d-1},\eta_{d-2})$ is a \emph{Bogomolov pair with respect to stable vector bundles} if it holds that $\int_X \Delta(E)\cdot \eta_{d-2}\ge 0$ for all vector bundles $E$ that are stable with respect to $\eta_{d-1}$. 
\item The classical Bogomolov inequality \cite{Bogomolov_holomorphictensors,Gieseker_Bogomolov} is that if $h$ is the class of an ample divisor on $X$ then $(h^{d-1},h^{d-2})$ is a Bogomolov pair.
\end{enumerate}
\end{remark}

\begin{conjecture}\label{conj:HRimpliesBG}

 If $(\eta_{d-1},\eta_{d-2})$ is a Hodge-Riemann pair and $\eta_{d-1}\in \Amp^{d-1}(X)$ (resp. $\eta_{d-1}\in \cK^{d-1}(X)$), then $(\eta_{d-1},\eta_{d-2})$ is a Bogomolov pair.
\end{conjecture}

Clearly \cref{conj:HRimpliesBG} holds when $X$ is the complex projective space $\P^d$. Below we present further examples supporting the conjecture.

\begin{lemma}\label{lem:reflexivestable}
Let $X$ be a complex smooth projective variety (resp.\ compact K\"ahler manifold), and let $\eta_{d-2}\in \Amp^{d-2}(X)$ and $\eta_{d-1}\in \Amp^{d-1}(X)$ (resp.\ $\eta_{d-2}\in \cK^{d-2}(X)$ and $\eta_{d-1}\in \cK^{d-1}(X)$).

Assume $(\eta_{d-1},\eta_{d-2})$ is a Hodge-Riemann pair.  If $(\eta_{d-1},\eta_{d-2})$ is a Bogomolov pair with respect to stable reflexive sheaves then it is a Bogomolov pair (i.e.\ a Bogomolov pair with respect to semistable sheaves).
\end{lemma}
\begin{proof}
We first deal with the case when $E$ is an $\eta_{d-1}$-stable torsion-free sheaf on $X$. Consider the short exact sequence
\[
    0 \to E \to E^{\vee \vee} \to E^{\vee \vee}/E \to 0.
\]
The double dual $E^{\vee \vee}$ is reflexive, $\eta_{d-1}$-stable, and the support of $E^{\vee \vee}/E$ is of codimension at least $2$. Thus $c_1(E)=c_1(E^{\vee \vee})$ and $$c_2(E^{\vee \vee}) = c_2(E) + c_2(E^{\vee \vee}/E) = c_2(E) - \sum\text{length}_Z(E^{\vee \vee}/E)[Z],$$  
where the sum is taken over all irreducible components of codimension $2$ of the support of $E^{\vee \vee}/E$. So as $\eta_{d-2}$ is positive we obtain
\[
    \Delta(E) \cdot \eta_{d-2}\ge \Delta(E^{\vee \vee}) \cdot \eta_{d-2} \ge 0.
\]

Next we consider the general case when $E$ is $\eta_{d-1}$-semistable. We argue by induction on the rank of $E$. The rank $1$ case is immediate since $E$ will be stable. Suppose now that $E$ is properly $\eta_{d-1}$-semistable and let $F \subset E$ be a proper saturated subsheaf such that $\mu_{\eta_{d-1}}(F) =\mu_{\eta_{d-1}}(E)$.

Then $F$ and $E/F$ are $\mu_{\eta_{d-1}}$-semistable and setting
$$\xi := \frac{c_1(F)}{\rk(F)} - \frac{c_1(E/F)}{\rk(E/F)},$$
we have
\begin{equation}\label{eq:bogomolovexact}
    -\frac{\rk(F)\rk(E/F)}{\rk(E)}\xi^2 = \frac{\Delta(E)}{\rk(E)} - \frac{\Delta(F)}{\rk(F)} - \frac{\Delta(E/F)}{\rk(E/F)}.
\end{equation}
Then since $\xi \cdot \eta_{d-1} = 0$, we get from the Hodge-Riemann property of $(\eta_{d-1},\eta_{d-2})$ that  $-\xi^2 \cdot \eta_{d-2} \ge 0$.   Moreover, by induction $\Delta(F) \cdot \eta_{d-2} \ge 0$ and $\Delta(E/F) \cdot \eta_{d-2} \ge 0$, and the conclusion follows. 
\end{proof}

\begin{lemma}\label{lem:limitBog}
Still assume that $\eta_{d-1}\in \Amp^{d-1}(X)$ and $\eta_{d-2}\in \Amp^{d-2}(X)$.    Suppose $(\eta_{d-1},\eta_{d-2})$ is a limit of pairs $(\eta_{d-1}(\epsilon),\eta_{d-2}(\epsilon))$ as $\epsilon\to 0$ where each $(\eta_{d-1}(\epsilon),\eta_{d-2}(\epsilon))$ is both a Bogomolov and Hodge-Riemann pair.   Then $(\eta_{d-1},\eta_{d-2})$ may not be a Hodge-Riemann pair but will still be a Bogomolov pair.  
\end{lemma}
\begin{proof}
Note first the proof of Lemma \ref{lem:reflexivestable} still applies to $(\eta_{d-1},\eta_{d-2}$) meaning it is sufficient to consider only reflexive sheaves $E$ that are $\eta_{d-1}$-stable, and such $E$ will be stable with respect to $\eta_{d-1}(\epsilon)$ for sufficiently small $\epsilon$ by Lemma \ref{lem:stabilityisopen}.  Then  $\int_X \Delta(E)\cdot \eta_{d-2}(\epsilon)\ge 0$ and we can let $\epsilon\to 0$.

See \cref{ex:nonHRlimit} for a situation where a limit of Hodge-Riemann pairs is not a Hodge-Riemann pair.
\end{proof}

\begin{remark}\label{rem:OpennessBogomolov}
 Let $X$ be a complex smooth projective variety (resp.\ compact K\"ahler manifold). 
 Let $\eta_{d-2}\in \Amp^{d-2}(X)$ and $\eta_{d-1}\in \Amp^{d-1}(X)$ (resp.\ $\eta_{d-2}\in \cK^{d-2}(X)$ and $\eta_{d-1}\in \cK^{d-1}(X)$) such that $(\eta_{d-1},\eta_{d-2})$ is a Hodge-Riemann pair.
 If $(\eta_{d-1},\eta_{d-2})$ is a Bogomolov pair and  $(\eta',\eta_{d-2})$ is some other Hodge-Riemann pair, then $(\eta',\eta_{d-2})$ is also a Bogomolov pair.
\end{remark}
\begin{proof}
     Let $E$ be a semistable torsion-free sheaf with respect to $\eta'$. We argue by induction on the rank $r$ of $E$. If $r=1$ the assertion is clear. Take now $r>1$. 
    By hypothesis the classes $\eta_{d-1}/\eta_{d-2}$ and $\eta'/\eta_{d-2}$ belong to the quadratic cone $\Pos_{\eta_{d-2}}(X)$ and thus the entire segment $[\eta_{d-1}/\eta_{d-2},\eta'/\eta_{d-2}]$ is included in $\Pos_{\eta_{d-2}}(X)$.
    If $E$ is also semistable with respect to $\eta_{d-1}$ we have  $\int_X \Delta(E)\cdot \eta_{d-2}\ge 0$ since $(\eta_{d-1},\eta_{d-2})$ is a Bogomolov pair. Otherwise there exists a class $\alpha\in[\eta_{d-1}/\eta_{d-2},\eta'/\eta_{d-2}]$ such that  $E$ is properly semistable with respect to $\alpha\eta_{d-2}$. We apply now the formula \eqref{eq:bogomolovexact} and the same argument as in the second part of the proof of  \cref{lem:reflexivestable} where stability is now taken with respect to  $\alpha\eta_{d-2}$.
\end{proof}

\begin{proposition}\label{prop:complete_intersections}
    Let $X$ be a complex smooth projective variety. Then for any classes $$\alpha_1,\ldots,\alpha_{d-2}\in\Amp^1(X) \text{ and } \eta\in(\alpha_1\cdots\alpha_{d-2}\Pos_{\alpha_1\cdots \alpha_{d-2}}(X))\cap\Amp^{d-1}(X)$$ the pair $(\eta,\alpha_1 \cdots \alpha_{d-2})$ is a Bogomolov pair. In other words, Conjecture \ref{conj:HRimpliesBG} holds when $\eta_{d-2}$ is a product of classes in $\Amp^1(X)$, and in particular when $X$ is a projective threefold.
\end{proposition}

\begin{proof}

When the classes $\alpha_1,\ldots,\alpha_{d-2}\in\Amp^1(X)$ are rational, the result is a consequence of \cite[Theorem 3.4]{Langer}. In general, we can approximate the classes $\alpha_i$ by rational ones and use \cref{lem:limitBog}.
\end{proof}

We now turn to corresponding definitions for differential forms:

\begin{definition}
Let $E$ be a torsion free sheaf.  If $h$ is a hermitian metric on the locally free locus of $E$ the 
 \emph{discriminant} of $(E,h)$ is the form
$$\Delta(E,h) = 2\rk(E) c_2(E,h)  - (\rk(E)-1) c_1(E,h)^2$$
where $c_i(E,h)$ denotes the $i$-th Chern form of $h$.
\end{definition}

\begin{definition}[Bogomolov pairs of differential forms]\label{def:bogomolovpairsofdifferentialforms}
 Let $X$ be a compact complex manifold of dimension $d$ and let $\Omega_{d-1}$, $\Omega_{d-2}$ be $\partial \bar{\partial}$-closed forms of types $(d-1,d-1)$ and $(d-2,d-2)$ respectively, and assume $\Omega_{d-2}$ is strictly weakly positive. 

    We say $(\Omega_{d-1},\Omega_{d-2})$ is a \emph{Bogomolov pair for stable vector bundles} if whenever $h$ is a hermitian metric on a locally free sheaf $E$ that satisfies the weak Hermitian-Einstein equation
    \begin{equation}i F_h \wedge \Omega_{d-1} = f \Omega_d \Id_E\label{eq:weakHE}\end{equation}
    where $F_h$ is the curvature of the Chern connection associated to $h$, $\Omega_d$ is a volume form and $f\in C^{\infty}(X)$, it holds that $$\Delta(E,h)\wedge \Omega_{d-2}\ge 0$$ pointwise over $X$. 
In a similar way we define {\em Bogomolov pairs for {stable} reflexive sheaves} if the above condition holds for all reflexive sheaves endowed with Hermitian-Einstein metrics which are {\em admissible} in the sense of \cite[Definition]{BandoSiu}. 
    
\end{definition}

\begin{remark}
\begin{enumerate}
\item Observe that if \eqref{eq:weakHE} holds then taking the trace and integrating we necessarily have
$$\int_X f\Omega_d = \frac{1}{r} \int_X c_1(E)\cdot [\Omega_{d-1}].$$
\item Assume that $\Omega_{d-1}$ is strictly  positive, and let $\omega = \sqrt[d-1]{\Omega_{d-1}}$ which is Gauduchon.    It turns out that if $E$ admits a weakly Hermitian-Einstein metric then after a conformal change one can find a hermitian metric that is Hermitian-Einstein (i.e. $iF_h\wedge \Omega_{d-1} = c\omega^d \Id$ where $c$ is constant over $X$) (apply \cite[2.1.5]{LuTe} with respect to the Guaduchon metric $\omega = \sqrt[d-1]{\Omega_{d-1}}$). 
\end{enumerate}
\end{remark}

\begin{remark}
Similarly one could make a definition of a Bogomolov pair of forms for semistable vector bundles using approximate Hermitian-Einstein metrics \cite[Chapter 6]{KobayashiVectorBundles}.  One can presumably also combine these to define a notion of Bogomolov pair for semistable reflexive sheaves using ``admissible approximate Hermitian-Einstein metrics'', but we are not aware of any work in which this has been considered even in the case $(\Omega_{d-1},\Omega_{d-2}) = (\omega^{d-1},\omega^{d-2})$ where $\omega$ is a K\"ahler form.
\end{remark}

A connection between Bogomolov pairs and Hodge-Riemann pairs of differential \textit{}forms is given by the following.

\begin{proposition}\label{prop:Bogomolov-HRpairs}
Let $(\Omega_{d-1},\Omega_{d-2})$ be a Hodge-Riemann pair.  Then
\begin{enumerate}
\item $(\Omega_{d-1},\Omega_{d-2})$ is a Bogomolov pair with respect to stable vector bundles.
\item Assume also that $\Omega_{d-1}$ is  strictly  positive.  Then $([\Omega_{d-1}],[\Omega_{d-2}])$ is a Bogomolov pair with respect to stable vector bundles. 
\item If moreover $\Omega_{d-1}=\omega^{d-1}$ for some K\"ahler form $\omega$ on $X$, then
 $(\Omega_{d-1},\Omega_{d-2})$ is a Bogomolov pair for stable reflexive sheaves and
$([\Omega_{d-1}],[\Omega_{d-2}])$ is a Bogomolov pair of classes.
\end{enumerate}
\end{proposition}
\begin{proof}

Suppose $h$ is any weakly Hermitian-Einstein metric, so satisfying \eqref{eq:weakHE}.  We set 
\[
    F_0 := F_h - \frac{1}{r}\tr(F_h)\cdot \Id_E.
\]
The Hermitian-Einstein condition translates to
\begin{equation}\label{eq:HEcond}
     F_0 \wedge \Omega_{d-1} = 0.
\end{equation}
A direct computation shows that
\[
    \Delta(E,h) :=2r c_2(E,h)  - (r-1) c_1(E,h)^2= \frac{r}{4\pi^2}\tr(F_0^2),
\]
see proof of \cite[Theorem 2.2.3]{LuTe}. Pointwise we get
\[
    \Delta(E,h) \wedge \Omega_{d-2} = \frac{r}{4\pi^2} \sum_{i=1}^r\sum_{j=1}^r F_{0,ij} \wedge F_{0,ji} \wedge \Omega_{d-2},
\]
where $(F_{0,ij})$ is the matrix corresponding to $F_0$ with respect to an $h$-unitary basis. Each term of the above sum is non-negative. Indeed, for $i = j$, $F_{0,jj}$ is purely imaginary since the matrix $(F_{0,ij})$ is anti-selfadjoint, it satisfies equation \eqref{eq:HEcond} and thus the Hodge-Riemann property of the pair $(\Omega_{d-1},\Omega_{d-2})$ gives $F_{0,jj}^2 \wedge \Omega_{d-2} \ge 0$. For $i \neq j$, we write $F_{0,ij} = \alpha + i\beta$ with $\alpha$ and $\beta$ real $(1,1)$-forms, and we get
\[
    F_{0,ij} \wedge F_{0,ji} \wedge \Omega_{d-2} = -(\alpha + i\beta) \wedge (\alpha - i\beta) \wedge \Omega_{d-2} = -(\alpha^2 + \beta^2) \wedge \Omega_{d-2} \ge 0
\]
by the same argument.  This proves 
\begin{equation}\Delta(E,h)\wedge \Omega_{d-2}\ge 0\label{eq:discriminantpositivepointwise}
\end{equation}
pointwise over $X$ which is what we wanted.

The second statement follows from the Hitchin-Kobayashi correspondence \cite{LiYau}, which says that if $E$ is a vector bundle that is stable with respect to $[\Omega_{d-1}]$ then it admits a Hermitian-Einstein metric with respect to the Gauduchon metric $\omega= \sqrt[d-1]{\Omega_{d-1}}$. Integrating \eqref{eq:discriminantpositivepointwise} over $X$ gives the result we want.

The same argument applies to the third statement. The Hitchin-Kobayashi correspondence holds also in this situation by work of Bando and Siu \cite{BandoSiu}.
\end{proof}
\begin{corollary}\label{cor:tori}  
    Conjecture \ref{conj:HRimpliesBG} holds for complex tori.
\end{corollary}
\begin{proof}
    If $X=\C^d/\Gamma$ is a complex torus and $(\eta_{d-1},\eta_{d-2})$ is a Hodge-Riemann pair on $X$ with $\eta_{d-1}\in\cK^{d-1}(X)$, then we may choose translation invariant representatives $\Omega_{d-1}$, $\Omega_{d-2}$  of $\eta_{d-1}$, $\eta_{d-2}$. It is clear that these representatives have the required pointwise positivity properties and in particular that there exists a translation invariant positive $(d-1)$-root $\omega$ of $\Omega_{d-1}$ which is therefore a K\"ahler form on $X$. We apply now Proposition \ref{prop:Bogomolov-HRpairs}. 
\end{proof}

\begin{remark}\label{rem:projectively-flat}
Let $(\Omega_{d-1},\Omega_{d-2})$ be a Hodge-Riemann pair of forms and let $E$ be a locally free sheaf on $X$ admitting a weakly Hermitian-Einstein metric $h$ with respect to $\Omega_{d-1}$. If one has equality in 
\[
    \Delta(E,h)\wedge \Omega_{d-2}\ge 0,
\]
then the Hodge-Riemann property of $(\Omega_{d-1},\Omega_{d-2})$ yields that the trace-free part $F_0$ of the Chern curvature of $(E,h)$ vanishes, and thus $\Delta(E,h) = 0$ and 
\begin{equation}\label{eq:projFlat}
      F_h = \frac{1}{r}\tr(F_h)\Id_E.
\end{equation}
This means that $(E,h)$ is projectively flat, or equivalently, $\P(E) \to X$ is induced by a unitary representation $\pi_1(X) \to \mathrm{PU}(r)$ (see \cite[Proposition 1.4.22]{KobayashiVectorBundles}). 

From $\eqref{eq:projFlat}$ one obtains that $(E,h)$ satisfies the weak Hermitian-Einstein condition \eqref{eq:weakHE} with respect to any $\partial \bar{\partial}$-closed strictly  positive form $\Omega_{d-1}'$ of type $(d-1,d-1)$. Thus, by the Hitchin-Kobayashi correspondence, $E$ is also $[\Omega'_{d-1}]$-polystable (see \cite[Theorem 2.3.2]{LuTe}). 

More generally, one may ask the following:
\begin{enumerate}
    \item[(Q)] Given a Hodge-Riemann pair $(\eta_{d-1},\eta_{d-2})$ of classes and an $\eta_{d-1}$-stable locally free sheaf $E$ such that $\Delta(E) \cdot \eta_{d-2} = 0$, then is $E$ projectively flat?
\end{enumerate}
For Hodge-Riemann pairs of the form $(\omega^{d-1},\omega^{d-2})$, with $\omega \in \cK^1(X)$, the answer is known via the Hitchin-Kobayashi correspondence. Therefore the statement holds when $X$ is a surface, however it remains open in higher dimensions.  
\end{remark}

\begin{remark}
We think that if $(\Omega_{d-1},\Omega_{d-2})$ is a Hodge-Riemann pair of differential forms on a compact K\"ahler manifold $X$ with $\Omega_{d-1}$ strictly positive we could deduce that $([\Omega_{d-1}],[\Omega_{d-2}])$ is a Bogomolov pair from a version of the proof of \cref{prop:Bogomolov-HRpairs} if we had a stronger form of the Hitchin-Kobayashi correspondence for reflexive sheaves.  One would need the existence of an admissible Hermitian-Einstein metric with respect to the Gauduchon metric $\omega = \sqrt[d-1]{\Omega_{d-1}}$ on any $[\Omega_{d-1}]$-stable reflexive sheaf.  Even then, this would not fully prove Conjecture \ref{conj:HRimpliesBG} since not every Hodge-Riemann pair $(\eta_{d-1},\eta_{d-2})$ can be written this way.
\end{remark}

\begin{example}\label{ex:nonHRlimit}
    In the situation of \cref{ex:delv} consider the pairs $(h^3,\eta+\epsilon h^2)$. We can check by a direct computation that they are Hodge-Riemann pairs with respect to $H^{1,1}(X)$ for $\epsilon>0$. Therefore they are Bogomolov pairs for $\epsilon\ge0$ by \cref{cor:tori} and \cref{lem:limitBog}. However the limit pair $(h^3,\eta)$ is not a Hodge-Riemann pair with respect to $H^{1,1}(X)$ as we know from \cref{ex:delv}. 
\end{example}

\section{Schur polynomials of K\"ahler classes}\label{sec:SchurPoly}

The following is a variant of a result of Chen \cite{Chen}.

\begin{theorem}\label{thm:HRimpliesBG}
Suppose $X$ is compact of dimension $d$.  Let $\alpha_1,\ldots,\alpha_e$ be K\"ahler classes on $X$ with $e\ge d-1$.  Let $\lambda$ be a partition of length $d-1$.  Then 
$$(s_{\lambda}(\alpha_1,\ldots,\alpha_e),s'_{\lambda}(\alpha_1,\ldots,\alpha_e))$$
is a Bogomolov pair. 
\end{theorem}

\begin{proof}
By Proposition \ref{prop:ShurHRpair} and Remark \ref{rem:diffformpairisacohomologypair}, $(s_{\lambda}(\alpha_1,\ldots,\alpha_e),s'_{\lambda}(\alpha_1,\ldots,\alpha_e))$ is a Hodge-Riemann pair.  So by Lemma \ref{lem:reflexivestable} we need only consider the case when $E$ is $s_{\lambda}(\alpha_1,\ldots,\alpha_e)$-stable and reflexive. 

For $i=1,\ldots,e$ let $\omega_i$ be a K\"ahler form in $\alpha_i$.   By \cite{Rossi} and \cite{Wlodarczyk}, there exists a proper modification $p: \hat{X} \to X$ with $\hat{X}$ smooth such that 
\begin{itemize}
    \item the induced morphism $\hat{X} \setminus p^{-1}(\Sing(E)) \to X \setminus \Sing(E)$ is an isomorphism, and
    \item $\hat{E} := p^*(E)/\Tors(p^*(E))$ is locally free.
\end{itemize}
Let $\theta$ be a K\"ahler form on $\hat{X}$, consider the forms $\hat{\omega}_{j,\varepsilon} := p^*\omega_j + \varepsilon \theta$ for $\varepsilon \ge 0$, and let $\hat{s}_{\lambda,\varepsilon} := s_\lambda(\hat{\omega}_{1,\varepsilon},\ldots,\hat{\omega}_{e,\varepsilon})$ and similarly for $\hat{s}'_{\lambda,\varepsilon}$. We also write $s_\lambda := s_\lambda(\omega_1,\ldots,\omega_e)$ and $s'_\lambda := s'_\lambda(\omega_1,\ldots,\omega_e)$ for simplicity.   

For small $\varepsilon > 0$, we show that $\hat{E}$ is $[\hat{s}_{\lambda,\varepsilon}]$-stable, and for this we apply \cite[Corollary 6.10]{TomaLimitareaII}. By loc.\ cit.\ it is enough to check that $\hat{E}$ is pseudo-stable with respect to $\hat{s}_{\lambda,0}$, that is, for any proper saturated subsheaf $F \subset \hat{E}$ the following inequality holds
\[
    \frac{\int_{\hat{X}} c_1(F) \cdot [\hat{s}_{\lambda,0}]}{\rk(F)} < \frac{\int_{\hat{X}} c_1(\hat{E}) \cdot [\hat{s}_{\lambda,0}]}{\rk(\hat{E})}.
\]
Here the word ``pseudo'' stresses the fact that stability is considered with respect to $\hat{s}_{\lambda,0}$, which is positive, but not necessarily strictly positive on $\hat{X}$. The above inequality may be rewritten as
\[
    \frac{\int_X p_*(c_1(F))\cdot [s_\lambda]}{\rk(p_*F)} < \frac{\int_X p_*(c_1(\hat{E}))\cdot [s_\lambda]}{\rk(E)}.
\]
The validity of this inequality follows from the stability of $E$ with respect to $[s_\lambda]$, remarking that 
\[
    p_*(c_1(F)) = c_1(p_*F) \quad \text{and}\quad p_*(c_1(\hat{E})) = c_1(E). 
\]

Since for $\varepsilon>0$ we know $(\hat{s}_{\lambda,\varepsilon}, \hat{s}'_{\lambda,\varepsilon})$ is a Hodge-Riemann pair by \cref{prop:ShurHRpair}, we have by \cref{prop:Bogomolov-HRpairs}
\[
    \Delta(\hat{E}) \cdot [\hat{s}'_{\lambda,\varepsilon}] \ge 0,
\]
for small $\varepsilon > 0$. Now letting $\varepsilon$ tend to $0$, we obtain
\[
    \Delta(\hat{E}) \cdot [\hat{s}'_\lambda] \ge 0.
\]
As before we have by the projection formula
\[
    \int_{\hat{X}} \Delta(\hat{E})\cdot p^*[s'_\lambda] = \int_X p_*(\Delta(E))\cdot [s'_\lambda],
\]
and also
\[
    p_*(\Delta(\hat{E})) = \Delta(E)
\]
since $E$ and $p_*(\hat{E})$ coincide in codimension $2$.

\end{proof}

\begin{remark}
As is clear from the proof of Theorem \ref{thm:HRimpliesBG},  we can replace the Schur polynomial $s_{\lambda}$ with any symmetric polynomial $p$ of degree $d-1$ such that $p'(\omega_1,\ldots,\omega_{e})$ has the Hodge-Riemann property for any K\"ahler forms $\omega_i$ on any complex manifolds $\hat{X}$ of dimension $d$.  Compare Remark \ref{rmk:otherpolynomials}.
\end{remark}

\section{Segre Classes of Ample Vector Bundles}\label{sec:SegreClasses}

Let $X$ be a  projective $d$-dimensional manifold, $A$ be a vector bundle of rank $e \ge d-1$ on $X$, $\pi : \P(A) \to X$ be the natural projection, and $\xi = c_1(\cO_{\P(A)}(1))$ be the Chern class of the tautological line bundle on $\P(A)$. For $h\in \Amp^1(X)$ and $t\in \mathbb R_{\ge 0}$ we set
$$\xi_t := \xi + t \pi^* h$$
Recall we say the $\mathbb R$-twisted vector bundle $A\langle th\rangle$ is ample if $\xi_t$ is ample on $\mathbb P(A)$.

\begin{proposition}\label{prop:ss-projective-bundle}
 Assume  $A$ is an ample vector bundle on $X$ satisfying
\begin{equation}\label{eq:ineqCond}
    s_d(A)\ge \mu_{\max,s_{d-1}(A)}(A),
\end{equation}
where $\mu_{\max,s_{d-1}(A)}(A)$ denotes the maximal slope with respect to $s_{d-1}(A)$ of non-trivial subsheaves of $A$.
   Let $E$ be an $s_{d-1}(A)$-semistable torsion-free sheaf on $X$. Then $\pi^*(E)$ is $\xi$-semistable on $\P(A)$. 

If moreover inequality \eqref{eq:ineqCond} is strict, then the pullback $\pi^*(E)$ of any $s_{d-1}(A)$-stable torsion-free sheaf $E$ on $X$ is  $\xi$-stable on $\P(A)$. 
\end{proposition}
\begin{proof}
We prove only the stable case, since the other one is completely similar. Suppose by contradiction that $E$ is an $s_{d-1}(A)$-stable torsion-free sheaf on $X$ but such that $\pi^*(E)$ is not $\xi$-stable on $\P(A)$. By passing to a suitable exterior power of $E$, we may assume that $\pi^*(E)$ is destabilized by an invertible subsheaf $\pi^*(L^{-1})(-k)$ of $\pi^*(E)$, where $L$ is a line bundle on $X$ and $k$ is some integer. This gives us a non-trivial section in $\pi^*(E\otimes L)(k)$ and shows in particular that $k>0$, since $\pi_*\cO_{\P(A)}(-n)=0$ for $n$ positive and $\pi_*\cO_{\P(A)}=\cO_X$. 
The fact that the section 
\begin{equation}\label{eq:destSection}
\cO_{\P(A)}\to \pi^*(E\otimes L)(k)   
\end{equation}
destabilizes implies 
$$k{s_d(A)}+\mu_{s_{d-1}(A)}(E \otimes L) \le 0.$$
Taking the pushforward of \eqref{eq:destSection} by $\pi$ we obtain a section
\[
    \cO_X \to E \otimes L \otimes S^kA
\]
over $X$, which gives
\[
    0 \le \mu_{s_{d-1}(A)}(E \otimes L) + \mu_{\max,s_{d-1}(A)}(S^kA) = \mu_{s_{d-1}(A)}(E \otimes L) + k\mu_{\max,s_{d-1}(A)}(A).
\]
The last two inequalities lead to
\[
    s_d(A) \le \mu_{\max,s_{d-1}(A)}(A),
\]
contradicting the hypothesis on $A$.
\end{proof}

\begin{remark}
\begin{enumerate}
    \item The inequality \eqref{eq:ineqCond} is satisfied by ample vector bundles $A$ which are destabilized by a rank one {\em sub-bundle} $A_1$ which realizes the maximal slope with respect to $s_{d-1}(A)$. Indeed, denoting by $A_2$ the quotient $A/A_1$, we can write in this case
    \begin{align*}
        &s_d(A)-\mu_{\max,s_{d-1}(A)}(A)= \\
        &= \sum_{i=0}^d s_i(A_1)s_{d-i}(A_2)-s_1(A_1)\sum_{i=0}^{d-1} s_i(A_1)s_{d-i}(A_2) \\
     &= s_d(A_2)+\sum_{j=1}^d(s_j(A_1)-s_1(A_1)s_{j-1}(A_1))s_{d-j}(A_2)\\
     &= s_d(A_2)>0.
    \end{align*}
    In particular, the inequality \eqref{eq:ineqCond} is satisfied when $A$ is a direct sum of ample line bundles. 
    \item
   The examples provided by Kleiman in  \cite[Example 5]{Kleiman} show that there exist rank two globally generated ample vector bundles on surfaces which do not satisfy the inequality \eqref{eq:ineqCond}. At least for $X=\P^2$, one easily sees that Kleiman's examples are all $s_{d-1}(A)$-unstable. It would be therefore interesting to know if a supplementary stability assumption with respect to $s_{d-1}(A)$ on $A$ could lead to the inequality \eqref{eq:ineqCond} for $A$.  
\end{enumerate}
\end{remark}

\begin{corollary}\label{cor:segrebogomolovpairI}
Assume that $A$ is ample and satisfies inequality \eqref{eq:ineqCond}. Then $$(s_{d-1}(A),s_{d-2}(A))$$ is a Bogomolov pair.
\end{corollary}
\begin{proof}
 That $(s_{d-1}(A),s_{d-2}(A))$ is a Hodge-Riemann pair  follows from \cref{prop:ShurAmpleHRpair} (this can also be proved using the Hodge-Riemann property for $\xi$ on $\mathbb P(A)$ and pushing forward to $X$). 

Now let $E$ be a torsion-free sheaf on $X$ that is semistable with respect to $s_{d-1}(A)$. By \cref{prop:ss-projective-bundle}, $\pi^*E$ is semistable with respect to $\xi$, so applying the classical Bogomolov inequality gives
$$\int_X \Delta(E) s_{d-2}(A) = \int_{\mathbb P(A)} \Delta(\pi^* E)\xi^{d+e-2}\ge 0.$$
\end{proof}

The following provides a partial answer to question (Q) posed in \cref{rem:projectively-flat} for Bogomolov pairs of type $(s_{d-1}(A),s_{d-2}(A))$.

\begin{corollary}\label{cor:eqBogomolov-forSegre}
Assume that $A$ is ample and that the inequality \eqref{eq:ineqCond} holds strictly. If $E$ is a locally free sheaf that is stable with respect to $s_{d-1}(A)$ and
\[
    \int_X \Delta(E) s_{d-2}(A) = 0,
\]
then $E$ is projectively flat. 
\end{corollary}
\begin{proof}
By \cref{prop:ss-projective-bundle}, we know that $\pi^*(E)$ is stable with respect to the ample class $\xi$ on $\P(A)$ and
\[
    \int_Y \Delta(\pi^*(E)) \cdot \xi^{d+e-2} = 0.
\]
Then $\pi^*(E)$ is projectively flat \cite[Theorem 4.4.7]{KobayashiVectorBundles}, so it corresponds to a unitary representation $\pi_1(\P(A)) \to \mathrm{PU}(r)$. Since $\pi_1(\P(A)) \to \pi_1(X)$ is an isomorphism \cite[Theorem 4.41]{HatcherAT}, it follows that $E$ is also projectively flat.
\end{proof}

\begin{lemma}\label{lem:stablityAth}
Assume that $A\langle th\rangle$ is ample for $0 < t \ll 1$, that $s_{d-1}(A)$ is an ample $(d-1)$-class, and that the inequality \eqref{eq:ineqCond} holds strictly for $A$. Then there exists some real number $0 < t_0$ such that for every $0 < t \le t_0$, if $E$ is an $s_{d-1}(A\langle th\rangle)$-stable torsion-free sheaf on $X$, then $\pi^*(E)$ is $\xi_t$-stable on $\P(A)$. 
\end{lemma}
\begin{proof}
The proof is similar to that of \cref{prop:ss-projective-bundle}. If $E$ is any $s_{d-1}(A\langle th\rangle)$-stable sheaf on $X$ such that $\pi^*(E)$ is not $\xi_t$-stable on $\P(A)$, then there exists a destabilizing section 
\[
    \cO_{\P(A)}\to \pi^*(E\otimes L)(k)   
\]
for some line bundle $L$ on $X$, as in \eqref{eq:destSection}, which gives
\[
    k \int_{\P(A)}\xi\cdot \xi_t^{d+e-2} + \mu_{s_{d-1}(A\langle th\rangle)}(E \otimes L) \le 0.  
\]
As before, this induces a section
\[
    \cO_X \to E \otimes L \otimes S^kA
\]
over $X$, from which we get
\[
    0 \le \mu_{s_{d-1}(A\langle th\rangle)}(E \otimes L) + k\mu_{\max,s_{d-1}(A\langle th\rangle)}(A).
\]
Putting together the last two inequalities we obtain
\[
    \int_{\P(A)}\xi\cdot \xi_t^{d+e-2} - \mu_{\max,s_{d-1}(A\langle th\rangle)}(A) \le 0.
\]

Now suppose that there is no such $0< t_0$ as in the statement. Therefore we find a decreasing sequence $(t_n)_{n \ge 1}$ that tends to 0, and such that
\begin{equation}\label{eq:tnIneq}
      \int_{\P(A)}\xi\cdot \xi_{t_n}^{d+e-2} - \mu_{\max,s_{d-1}(A\langle t_nh\rangle)}(A) \le 0.
\end{equation}
Consider the family of saturated subsheaves of $A$ which are (weakly) destabilizing for $A$ with respect to some class $s_{d-1}(A\langle th\rangle)$ with $t \in [0,t_1]$. This family is bounded, cf. \cite[Lemma 6.2]{MegyPavelToma}, and thus
$\mu_{\max,s_{d-1}(A\langle t_nh\rangle)}(A)$ tends to $\mu_{\max,s_{d-1}(A)}(A)$ as $t_n \to 0$. Letting $t_n \to 0$ in the inequality \eqref{eq:tnIneq}, we get
\[
    \int_{\P(A)}\xi^{d+e-1} = s_d(A) \le \mu_{\max,s_{d-1}(A)}(A),
\]
contradicting the supposed strict inequality in \eqref{eq:ineqCond}.
\end{proof}

\begin{corollary}\label{cor:segrebogomolovpair}
Assume that $A\langle th\rangle$ is ample for $0 < t \ll 1$, that $s_{d-1}(A)$ is an ample $(d-1)$-class, and that the inequality \eqref{eq:ineqCond} holds strictly for $A$. Then $$(s_{d-1}(A\langle th\rangle),s_{d-2}(A\langle th\rangle))$$ is a Bogomolov pair for $0 \le t \ll 1$.
\end{corollary}
\begin{proof}
By \cref{lem:reflexivestable}, it suffices to prove the Bogomolov inequality for reflexive sheaves that are $s_{d-1}(A\langle th\rangle)$-stable on $X$. For $0 < t \ll 1$, this follows by using \cref{lem:stablityAth} together with the classical Bogomolov inequality with respect to $\xi_t$ on $\P(A)$.

If $E$ is an $s_{d-1}(A)$-stable reflexive sheaf on $X$, then by \cref{lem:stabilityisopen} it is also $s_{d-1}(A\langle th\rangle)$-stable on $X$ for $0 < t \ll 1$. Hence
\[
    \int_X \Delta(E) \cdot s_{d-2}(A\langle th\rangle) \ge 0.
\]
Letting $t \to 0^+$, we get the Bogomolov inequality for $E$ with respect to $s_{d-2}(A)$.
\end{proof}

\section{Globally Generated Ample Bundles}

\begin{theorem}\label{thm:globallygenerated_chern}
Let $X$ be a projective manifold of dimension $d\ge 2$.  Suppose that $A$ is a globally generated and ample vector bundle of rank at least $d-1$ on $X$, which moreover satisfies
\begin{equation}\label{eq:ineqChern}
     c_d(A) > \mu_{\max,c_{d-1}(A)}(A).
\end{equation}
Then
$$(c_{d-1}(A),c_{d-2}(A))$$
is a Bogomolov pair.
\end{theorem}

\begin{proof}
By \cref{prop:ShurAmpleHRpair} $(c_{d-1}(A),c_{d-2}(A))$ is a Hodge-Riemann pair, so Lemma \ref{lem:reflexivestable} says is sufficient to prove that $(c_{d-1}(A),c_{d-2}(A))$ is a Bogomolov pair with respect to $c_{d-1}(A)$-stable reflexive sheaves.   
To this end, let $E$ be reflexive and $c_{d-1}(A)$-stable.  As $A$ is globally generated there is an exact sequence
$$ 0 \to K \to \mathcal O_X^N \to A\to 0$$
for some $N\in \mathbb N$.  If $c$ denotes the total Chern class, and $s$ the total Segre class, we then have
$ c(A) c(K) = 1$, which gives
$$ c(A) = s(K^\vee).$$

Let $h$ be an ample class on $X$. The surjection $\mathcal O_X^N\to K^\vee\to 0$ shows that $K^\vee$ is a nef bundle, and hence $K^\vee\langle th\rangle$ is ample for $0<t\ll 1$. Moreover, $s_{d-1}(K^\vee) = c_{d-1}(A)$ is an ample $(d-1)$-class, and the inequality \eqref{eq:ineqChern} becomes
\[
    s_d(K^\vee) > \mu_{\max,s_{d-1}(K^\vee)}(K^\vee).
\]
Hence by  \cref{cor:segrebogomolovpair}
$$\int_X \Delta(E) s_{d-2}(K^\vee) \ge 0,$$
which completes the proof.
\end{proof}

\begin{remark}
We expect that the above proof can be improved to show that $(s_{\lambda}(A),s_{\lambda}'(A))$ is a Bogomolov pair when $A$ is ample, globally generated, and satisfies the inequality \eqref{eq:ineqChern} (for one can use the cone construction of Fulton-Lazarsfeld as described in \cite{RossToma1} to relate $(s_{\lambda}(A),s_{\lambda}'(A))$ to a pair of Chern classes of a globally generated bundle on a certain normal variety $C$, which would then require an extension of several of the results in this paper to allow our base $X$ to be normal rather than smooth).
\end{remark}
\begin{remark}
It seems likely that neither the hypothesis that $A$ is globally generated nor the inequality \eqref{eq:ineqChern} is needed, but we do not know how to prove this.
\end{remark}

\section{Boundedness of Semistable Sheaves}

Let $X$ be either a complex smooth projective variety or a compact K\"ahler manifold of dimension $d$. In this section we establish several boundedness statements for semistable torsion-free sheaves on $X$.

\begin{definition}\label{def:boundedFam}
A set $\cS$ of isomorphism classes of coherent sheaves on $X$ is called \textit{bounded} if 
\begin{enumerate}
    \item \textbf{Algebraic case:} there exists a scheme $S$ of finite type over $\C$ and  a coherent sheaf $E$ on $S \times X$ such that $\cS$ is contained in the set of isomorphism classes of fibers of $E$ over points of $S$ \cite[Definition 1.7.5]{HL}.
    \item \textbf{Analytic case:} there exists a complex analytic space $S$, a compact subset $K \subset S$ and a coherent sheaf $E$ on $S \times X$ such that $\cS$ is contained in the set of isomorphism classes of fibers of $E$ over points of $K$ \cite[Definition 5.1]{TomaLimitareaII}.
\end{enumerate}
\end{definition}

When $X$ is projective, the two definitions given above coincide by the GAGA Theorem, cf. \cite[Remark 3.3]{TomaLimitareaI}. The following notion will be useful for showing our boundedness results.

\begin{definition}\label{def:boundednessclass}
    A class $\eta' \in \Amp^{d-1}(X)$ (resp. $\eta' \in \cK^{d-1}(X)$) is called a {\em boundedness class} if the following boundedness criterion holds:
    
    A set of isomorphism classes of torsion-free sheaves $E$ on $X$ is bounded if and only if the rank and the Chern classes of the sheaves $E$ take finitely many values, and their maximal slope $\mu_{\max,\eta'}(E)$ with respect to $\eta'$ is bounded above.
\end{definition}

In the projective setup it is known that any class of the form $\omega^{d-1} \in \Amp^{d-1}(X)$ with $\omega \in \Amp^1(X)$ satisfies the boundedness criterion; see \cite[Theorem 3.3.7]{HL} for the case where $\omega$ is rational, and \cite[Proposition 7.20]{joyce2021enumerative} for the general case.

\begin{proposition}\label{prop:boundedness_pairs}
Let $X$ be projective and $\alpha_1,\ldots,\alpha_{d-2}\in\Amp^1(X)$ be rational. Then any $\eta_{d-1}$ inside
    $$\Amp^{d-1}(X)\cap (\alpha_1 \cdots \alpha_{d-2}\Pos_{\alpha_1 \cdots \alpha_{d-2}}(X))$$ is a boundedness class.
\end{proposition}
\begin{proof}
This is proved in \cite[Theorem 1.4]{PavelRossToma} (the technique used there is to restrict to hyperplane sections, which is why the hypothesis that the $\alpha_i$ be rational is required).
\end{proof}

The following result and its proof are a generalization of \cite[Proposition 6.3]{GrebToma}.
\begin{theorem}\label{thm:MainBound}
Let $X$ be either a complex smooth projective variety or a compact K\"ahler manifold of dimension $d$. Let $K \subset \Amp^{d-1}(X) \times \Amp^{d-2}(X)$ (resp. $K \subset \cK^{d-1}(X) \times \cK^{d-2}(X)$ in the K\"ahler case) be a path-connected compact subset, and denote by $K' := \pr_1(K)$ and $K'' := \pr_2(K)$ its two corresponding projections. Suppose that

\begin{enumerate}
    \item there exists a pair $(\eta',\eta'') \in K$ such that $\eta'$ is a boundedness class and
    \item for every $\eta_{d-1} \in K'$ there exists $\eta_{d-2} \in K''$ and a path $$\gamma_{(\eta_{d-1},\eta_{d-2})} : [0,1] \to K$$ connecting $(\eta_{d-1},\eta_{d-2})$ to $(\eta',\eta'')$ such that for all $t \in [0,1]$ the pair $\gamma_{(\eta_{d-1},\eta_{d-2})}(t)$ is a Hodge-Riemann and Bogomolov pair. 
\end{enumerate}
Then the set $\Sigma$ of isomorphism classes of torsion-free sheaves of fixed rank $r$ and fixed Chern classes $c_i \in N^i(X)$ (resp. in $H^{2i}(X,\Z)$) that are $\eta_{d-1}$-semistable with respect to some $\eta_{d-1}\in K'$ is bounded.
\end{theorem}
\begin{proof}
We prove the statement when $X$ is projective; the K\"ahler case is identical. We first introduce some notation. Fix $h\in \Amp^1(X)$ and consider the cone bundle 
$$\cC:=\left\{(\eta\cdot \alpha, \eta)\in  N^{d-1}(X)\times K'' \ \mid \ \alpha\in N^1(X), \ \int_X \alpha^2\cdot\eta>0, \ \int_X h\cdot \alpha\cdot\eta>0\right \}$$ 
inside the trivial real vector bundle $N^{d-1}(X)\times K''$ over $K''$. Consider also the vector sub-bundle 
$$S:=\{ ((\eta_{d-1},\eta_{d-2}),\zeta)\in \cC\times N^1(X) \ \mid \ \zeta \cdot \eta_{d-1}=0\}$$
of the trivial real vector bundle $\cC\times N^{1}(X)$ over $\cC$
and the following metric in the fibers of $S$, 
$$ \| ((\eta_{d-1},\eta_{d-2}),\zeta)\|_S^2:=-\zeta^2\cdot\eta_{d-2}.$$
This defines a metric indeed since $(\eta_{d-1},\eta_{d-2})$ is a Hodge-Riemann pair for each $(\eta_{d-1},\eta_{d-2}) \in \cC$. Note that $K\subset \cC$ by assumption. 

We fix a norm $\| \cdot \|_{N^1(X)}$ on $N^1(X)$, that we use throughout. This gives us a metric on the trivial  bundle $K\times N^1(X)$ whose  restriction to $S|_{K}$ is comparable to $\|\cdot \|_{S}$ over $K$ since $K$ is compact, in particular there exists some $k>0$ such that $\| \zeta \|_{N^1(X)}\le k\|  (\eta,\zeta)\|_S$  for all $(\eta,\zeta)\in S|_{K}$.\\

Next we aim to show that for all torsion-free coherent sheaves $E$ whose isomorphism class $[E]$ lies in $\Sigma$, the maximal slope $\mu_{\max,\eta'}(E)$ is bounded above by a constant $C := C(r,c_1,c_2,K)$ depending only $r$, $c_1$, $c_2$ and $K$. The result then follows since $\eta'$ is assumed to be a boundedness class. 

If such a sheaf $E$ is already $\eta'$-semistable, then $\mu_{\max,\eta'}(E) = \mu_{\eta'}(E)$ is clearly upper bounded as required. So we may consider instead the family $\Sigma'$ of sheaves $E$ with $[E] \in \Sigma$ that are not $\eta'$-semistable.

Let $E$ be a torsion-free sheaf whose isomorphism class $[E]$ lies in $\Sigma'$ . By assumption (2), there exists a pair $(\eta_{d-1},\eta_{d-2}) \in K$ such that $E$ is $\eta_{d-1}$-semistable, and moreover there is a path $\gamma : [0,1] \to K$ connecting $(\eta_{d-1},\eta_{d-2})$ to $(\eta',\eta'')$ such that for all $t \in [0,1]$ the pair $\gamma(t)$ is a Hodge-Riemann and Bogomolov pair. 

We will show that $E$ admits a filtration
\begin{align*}
     0 = E_0 \subset E_1 \subset \cdots \subset E_m =E
\end{align*}
such that each factor $E_i/E_{i-1}$ is torsion-free, $\eta'$-semistable, and has its first Chern class bounded by some constant $C$ depending only on $r, c_1, c_2, K$. Here when we say we bound the first Chern class we mean in terms of the fixed norm $\| \cdot \|_{N^1(X)}$ on $N^1(X)$. Then the maximal $\eta'$-destabilizing subsheaf $F \subset E$ admits a nontrivial morphism to one of these factors, therefore $\mu_{\eta'}(F) = \mu_{\max,\eta'}(E) \le C$. 

We argue by contradiction, so suppose there is no such filtration as above. Consider the following statement:\\

$P(m)$: There is a partition $(r_1,\ldots,r_m)$ of $r$ with $0 < r_1, \ldots, r_m \le r$ such that $E$ admits a filtration $0 = E_0 \subset E_1 \subset \cdots \subset E_m =E$ whose factors $E_i/E_{i-1}$ are torsion-free of rank $r_i$. Moreover, there exists $t_0 \in [0,1]$ such that each factor $E_i/E_{i-1}$ is $\pr_1(\gamma(t_0))$-semistable, and has its first Chern class bounded by some constant depending only on $r, c_1, c_2, K$.\\

We prove by induction that $P(m)$ holds for $1 \le m \le r$. The case $m = 1$ is clear since $E$ is $\eta_{d-1} = \pr_1(\gamma(0))$-semistable and its first Chern class is fixed by hypothesis. Now suppose that $P(m)$ holds for some $m< r$. Then there exists $t_0 \in [0,1]$, a partition $(r_1,\ldots,r_m)$ of $r$ with $0 < r_1, \ldots, r_m \le r$, and a filtration $0 = E_1 \subset \ldots \subset E_m =E$ as in the statement of $P(m)$. In particular the factors $F_i := E_i/E_{i-1}$ are $\pr_1(\gamma(t_0))$-semistable. 

Consider
\[
    t_1 := \sup \{ t \in [t_0,1] \, \mid \, \text{all factors $F_i$ are $\pr_1(\gamma(t))$-semistable} \}.
\]
We may assume that $t_1 < 1$, since otherwise all factors are $\eta'$-semistable and we reach a contradiction. In this case one of the factors is properly $\pr_1(\gamma(t_1))$-semistable, say $F_{i_0}$ for some $1 \le i_0 \le m$. Thus there is a short exact sequence
\[
    0 \to F' \to F_{i_0} \to F'' \to 0
\]
such that $F'$ and $F''$ are $\pr_1(\gamma(t_1))$-semistable of ranks $r'$, respectively $r''$, and $$\mu_{\pr_1(\gamma(t_1))}(F') = \mu_{\pr_1(\gamma(t_1))}(F_{i_0}) = \mu_{\pr_1(\gamma(t_1))}(F'').$$
Set
\[
    \xi := \frac{c_1(F')}{r'} - \frac{c_1(F'')}{r''}, \quad \xi_{ij} := \frac{c_1(F_i)}{r_i} - \frac{c_1(F_j)}{r_j}
\]
for $1 \le i,j \le m$. Then a straightforward computation yields
\begin{align*}
        \frac{\Delta(E)}{r} ={}& \sum_{i=1}^m \frac{\Delta(F_i)}{r_i} - \frac{1}{r} \sum_{i < j} r_i r_j \xi_{ij}^2 \\
        ={}& \frac{\Delta(F')}{r'} + \frac{\Delta(F'')}{r''} + \sum_{i \neq i_0} \frac{\Delta(F_i)}{r_i} - \frac{1}{r} \sum_{i < j} r_i r_j \xi_{ij}^2  - \frac{r'r''}{r_{i_0}} \xi^2.
\end{align*}
We know that $\gamma(t_1)$ is a Hodge-Riemann and Bogomolov pair, that $\xi \cdot \pr_1(\gamma(t_1)) = 0$, 
and that the first Chern classes of the $F_i$ are bounded as in the statement of $P(m)$ (so the $\xi_{ij}$ are bounded).  So we obtain that $
-\xi^2 \cdot \pr_2(\gamma(t_1))$ is bounded from below by $0$ and bounded from above in terms of $r$, $c_1$, $c_2$ and $K$. This gives bounds on $\|\xi\|_{N^1(X)}$, $c_1(F')$ and $c_1(F'')$ depending only on $r$, $c_1$, $c_2$ and $K$. We thus obtain a new filtration 
\[
    0 = E_0 \subset \cdots\subset E_{i_0-1} \subset F \subset E_{i_0} \subset \cdots \subset E_m = E
\]
where $F$ is the preimage of $F'$ in $E_{i_0}$, which makes $P(m+1)$ hold.

Consequently $P(r)$ is valid, thus there exists a filtration 
\begin{align*}
     0 = E_0 \subset E_1 \subset \cdots \subset E_r =E
\end{align*}
such that each factor $E_i/E_{i-1}$ is torsion-free, of rank 1, and has its first Chern class bounded by some constant depending only on $r, c_1, c_2, K$. Since the factors have rank $1$, they are in particular $\eta'$-semistable, and we reached a contradiction.
\end{proof}

\begin{corollary}\label{cor:boundednessschurkahler}
Let $\tilde{K}$ be a path-connected  and compact set of K\"ahler classes on $X$ that includes a rational point, $\lambda$ is a partition of length $d-1$ and $e\ge d-1$. Then the set of isomorphism classes of torsion-free sheaves of given topological type that are semistable with respect to some element of 
$$  K'=\{ s_{\lambda}(\alpha_1,\ldots,\alpha_e) \, \mid \, \alpha_1,\ldots,\alpha_e\in \tilde{K}\}$$
is bounded.
\end{corollary}
\begin{proof}
Set 
$$  K''=\{ s'_{\lambda}(\alpha_1,\ldots,\alpha_e) \, \mid \, \alpha_1,\ldots,\alpha_e\in \tilde{K}\}.$$
By \cref{prop:ShurHRpair} and \cref{thm:HRimpliesBG} the set $K=K'\times K''$ satisfies the hypotheses of Theorem \ref{thm:MainBound} which gives the result we want.
\end{proof}

\begin{corollary}\label{cor:CIboundedness}
    Let $X$ be a complex smooth projective variety and let $K'$ be a compact subset of 
     $$\bigcup_{\alpha_1,\ldots,\alpha_{d-2}\in\Amp^1(X)} \hspace{-1cm} \alpha_1\cdots\alpha_{d-2}\Pos_{\alpha_1 \cdots \alpha_{d-2}}(X) \cap \Amp^{d-1}(X).$$ 
Then the set of isomorphism classes of torsion-free sheaves of given topological type that are $\eta_{d-1}$-semistable with respect to some $\eta_{d-1}\in K'$ is bounded.
\end{corollary}
\begin{proof}
The sets $\Amp^{d-1}(X)\cap(\alpha_1 \cdots \alpha_{d-2}\Pos_{\alpha_1 \cdots \alpha_{d-2}}(X))$ are open and vary continuously with the $\alpha_i$.   Hence we may assume that $K'$ is contained in a single set of the form $\Amp^{d-1}(X)\cap(\alpha_1 \cdots \alpha_{d-2}\Pos_{\alpha_1 \cdots \alpha_{d-2}}(X))$ where the $\alpha_i\in \Amp^1(X)$ are all rational. Taking
\[
    K'' := \{ \alpha_1 \cdots \alpha_{d-2} \} \subset \Amp^{d-2}(X),
\]
we obtain using \cref{prop:complete_intersections} and \cref{prop:boundedness_pairs} that the set $K = K' \times K''$ fulfills the conditions of \cref{thm:MainBound}, which proves the statement.
\end{proof}

The following example can also be found in \cite{MegyPavelToma}.

\begin{example}
    We consider the projectivized bundle $X=\P(E)$ over $\P^1$, where $E=\cO_{\P^1}\oplus \cO_{\P^1}\oplus\cO_{\P^1}(-1)$. The effective and the nef cones of $X$ were computed by Fulger and Lehmann in \cite[Example 3.11]{FulgerLehmann2017cones}. They found 
    $$ \overline{\Eff}^1(X) =  \langle\gf,\xi \rangle, \ \Nef^1(X) =  \langle\gf,\xi+ \gf \rangle, \
\overline{\Eff}^2(X) =  \langle\xi\gf,\xi^2 \rangle
, \ \Nef^2(X) =  \langle\xi\gf, \xi\gf+\xi^2  \rangle,$$
where $\gf$ is the class of the fiber of $X\to\P^1$, and $\xi$ is the class of $\cO_{\P(E)}(1)$, the relations between them being $\xi^3=-1$ and $\xi^2\gf=1$. 
From this one easily computes the cone of complete intersection curve classes $\CI^2(X)$ and finds that $\CI^2(X)\varsubsetneq\Amp^2(X)$. More precisely, 
$$ \overline{\CI}^2(X)=\langle\xi\gf,\xi\gf+\frac{1}{2}\xi^2\rangle.$$ 
One can also check that in this example one has 
$$ \Amp^2(X)=\bigcup_{\alpha\in\Amp^1(X)}\alpha\Pos_\alpha(X),$$
in particular it follows that for any compact subset $K'$ of $\Amp^2(X)$ the set of isomorphism classes of torsion-free sheaves of given topological type that are $\eta_{2}$-semistable with respect to some $\eta_{2}\in K'$ is bounded.
\end{example}

\section{Bogomolov Pairs for Higgs Sheaves}

In this section we show that Hodge-Riemann pairs also lead to Bogomolov inequalities for Higgs sheaves. Let $(X,\omega)$ be a compact complex K\"ahler manifold of dimension $d$. By definition, a Higgs sheaf $(E,\theta)$ on $X$ consists of a coherent sheaf $E$ on $X$ together with a holomorphic map $\theta : E \to E \otimes \Omega^1_X$, called the Higgs field, such that $\theta \wedge \theta = 0$. 

We next recall the notion of Hermitian-Yang-Mills metrics for Higgs bundles (see \cite[Section 3]{Simpson-HiggsBundles}). Let $(E,\theta)$ be a Higgs bundle on $X$, i.e.\ a Higgs sheaf with $E$ locally free. Given a hermitian metric $h$ on $E$, we define the adjoint $\overline{\theta}_h$ of $\theta$ by
\[
    (\theta u,v)_h = (u,\overline{\theta}_hv)_h.
\]
Let $D_h$ be the Chern connection of $E$ compatible with the holomorphic structure on $E$ and the hermitian metric $h$. Consider
\[
    D_{h,\theta} := D_h + \theta + \overline{\theta}_h,
\]
which is usually called the Hitchin-Simpson connection, and let $F_{h,\theta} = D_{h,\theta}^2$ be the Hitchin-Simpson curvature of $D_{h,\theta}$. If $F_h = D_h^2$ denotes the curvature of $D_h$, then one can compute
\[
    F_{h,\theta} = F_h + D_h \theta + D_h\overline{\theta}_h + [\theta, \overline{\theta}_h]
\]
and
\[
    F_{h,\theta}^{1,1} = F_h + [\theta, \overline{\theta}_h].
\]

We define the discriminant of the Higgs bundle $(E,\theta)$ with respect to the connection $D_{h,\theta}$ by
\[
    \Delta(E,D_{h,\theta}) := 2r c_2(E,D_{h,\theta})  - (r-1) c_1(E,D_{h,\theta})^2.
\]

\begin{definition}\label{def:HYM-Higgs}
Let $\Omega_{d-1}$ be a $\partial \bar{\partial}$-closed form of type $(d-1,d-1)$ on $X$. A hermitian metric $h$ on a Higgs bundle $(E,\theta)$ is called \textit{Hermitian-Yang-Mills (HYM) with respect to $\Omega_{d-1}$} if
\[
    i(F_h + [\theta, \overline{\theta}_h])\wedge \Omega_{d-1} = \lambda \omega^d \Id_E
\]
for some constant $\lambda$ (or, equivalently, $iF_{h,\theta} \wedge \Omega_{d-1} = \lambda \omega^d \Id_E$).
\end{definition}

\begin{definition}
Let $\eta_{d-1} \in \cK^1(X)$. A Higgs sheaf $(E,\theta)$ on $X$ is called $\eta_{d-1}$-semistable (resp. stable) if $E$ is torsion-free and
\[
    \mu_{\eta_{d-1}}(F) \le \mu_{\eta_{d-1}}(E) \quad \text{(resp. }<)
\]
for all proper Higgs subsheaves $F \subset E$ (i.e.\ subsheaves satisfying $\theta(F) \subset F \otimes \Omega^1_X$). Analogously to \cref{def:sstability} we also get a natural notion of polystability for Higgs sheaves.
\end{definition}

\begin{remark}\label{rmk:non-abelianHodge}
The existence of HYM metrics on Higgs bundles is related to the above notion of stability via the non-abelian Hodge correspondence \cite{Simpson-HiggsBundles}, \cite{Nie-Zhang-HiggsGauduchon}. More precisely, if $\Omega_{d-1}$ is strictly positive, then $(E,\theta)$ is an $[\Omega_{d-1}]$-polystable Higgs bundle if and only if $(E,\theta)$ admits a HYM metric with respect to $\Omega_{d-1}$. The positivity of $\Omega_{d-1}$ is important for ensuring the existence of a Gauduchon metric $\omega' = \sqrt[d-1]{\Omega_{d-1}}$, which in turn allows the application of \cite[Theorem 1.1]{Nie-Zhang-HiggsGauduchon} to $(X,\omega')$.
\end{remark}

The following Bogomolov inequality is a generalization of the classical one for Higgs bundles \cite[Proposition 3.4]{Simpson-HiggsBundles}; see also \cite[Corollary 3.4]{Chen-Wentworth-nonabelianHodge}.
\begin{proposition}\label{prop:BG-Higgs} 
Let $(\Omega_{d-1},\Omega_{d-2})$ be a Hodge-Riemann pair of forms. Then
\begin{enumerate}
    \item If $(E,\theta)$ is a Higgs bundle admitting a HYM metric $h$ with respect to $\Omega_{d-1}$, then
\[
    \Delta(E,D_{h,\theta}) \wedge \Omega_{d-2} \ge 0
\]
pointwise over $X$.
    \item Assume also that $\Omega_{d-1}$ is strictly positive. If $(E,\theta)$ is an $[\Omega_{d-1}]$-stable Higgs bundle, then
\[
    \int_X \Delta(E) \cdot [\Omega_{d-2}] \ge 0.
\]
\end{enumerate}
\end{proposition}
\begin{proof}
Let $(E,\theta)$ be a Higgs bundle on $X$ admitting a HYM metric $h$ with respect to $\Omega_{d-1}$ (see \cref{def:HYM-Higgs}). In particular
\[
    F_{h,\theta}^\perp \wedge \Omega_{d-1} = 0,
\]
where $F_{h,\theta}^\perp$ denotes the trace-free part of the Hitchin-Simpson curvature of $(E,\theta,h)$. Then, as in the proof of \cref{prop:Bogomolov-HRpairs}, one obtains
\[
    \tr(F_{h,\theta}^\perp \wedge F_{h,\theta}^\perp)\wedge \Omega_{d-2} \ge 0
\]
pointwise, since $(\Omega_{d-1},\Omega_{d-2})$ is Hodge-Riemann. By Chern-Weil theory we also have
\[
    \Delta(E,D_{h,\theta}) := 2r c_2(E,D_{h,\theta})  - (r-1) c_1(E,D_{h,\theta})^2 = \frac{\rk(E)}{4\pi^2} \tr(F_{h,\theta}^\perp \wedge F_{h,\theta}^\perp),
\]
see proof of \cite[Theorem 2.2.3]{LuTe}. Hence
\begin{equation}\label{eq:discriminantHiggs}
    \Delta(E,D_{h,\theta}) \wedge \Omega_{d-2} = \frac{\rk(E)}{4\pi^2} \tr(F_{h,\theta}^\perp \wedge F_{h,\theta}^\perp)\wedge \Omega_{d-2} \ge 0.
\end{equation}

The second statement follows by the non-abelian Hodge correspondence (see \cref{rmk:non-abelianHodge}) and from integrating \eqref{eq:discriminantHiggs} to obtain
\[
    \int_X \Delta(E) \cdot [\Omega_{d-2}] = \frac{\rk(E)}{4\pi^2} \int_X [\tr(F_{h,\theta}^\perp \wedge F_{h,\theta}^\perp)] \cdot [\Omega_{d-2}] \ge 0.
\]
\end{proof}

The following result generalizes \cref{thm:HRimpliesBG} to the case of Higgs sheaves.

\begin{proposition}
Under the notation of \cref{thm:HRimpliesBG}, $$(s_{\lambda}(\alpha_1,\ldots,\alpha_e),s'_{\lambda}(\alpha_1,\ldots,\alpha_e))$$
is a Bogomolov pair for Higgs sheaves, i.e.\ for any $s_{\lambda}(\alpha_1,\ldots,\alpha_e)$-semistable Higgs sheaf $(E,\theta)$ on $X$,
\[
    \int_X \Delta(E) \cdot s'_{\lambda}(\alpha_1,\ldots,\alpha_e) \ge 0.
\]
\end{proposition}
\begin{proof}[Sketch of proof]
The proof is similar to that of \cref{thm:HRimpliesBG} and uses the construction in \cite{Biswas-Schumacher-HiggsSheaves} (see also \cite[p.\ 466]{Cardona-HiggsBundles}). As before, it is enough to treat the case of an $s_{\lambda}(\alpha_1,\ldots,\alpha_e)$-stable reflexive Higgs sheaf $(E,\theta)$ on $X$. Consider a proper modification $p: \hat{X} \to X$ with $\hat{X}$ smooth such that 
\begin{itemize}
    \item the induced morphism $\hat{X} \setminus p^{-1}(\Sing(E)) \to X \setminus \Sing(E)$ is an isomorphism, and
    \item $\hat{E} := p^*(E)/\Tors(p^*(E))$ is locally free.
\end{itemize}
The composition 
\[
    p^*(E) \to p^*(E \otimes \Omega^1_X) \to p^*(E) \otimes \Omega^1_{\hat{X}}
\]
sends $\Tors(p^*(E))$ to $\Tors(p^*(E)) \otimes \Omega^1_{\hat{X}}$. Hence it will descend to the quotient $p^*(E)/\Tors(p^*(E)$ and define a Higgs field $\hat{\theta}$ on $\hat{E}$ satisfying $\hat{\theta} \wedge \hat{\theta} = 0$. 

Now we are in a situation where $(\hat{E},\hat{\theta})$ is a stable Higgs bundle on $\hat{X}$ with respect to $\hat{s}_{\lambda,\varepsilon}$ for small $\varepsilon > 0$ (here we use the same notation for $\hat{s}_{\lambda,\varepsilon}$ as in the proof of \cref{thm:HRimpliesBG}). By the non-abelian Hodge correspondence (see \cite{Nie-Zhang-HiggsGauduchon}), there is a HYM metric on $\hat{E}$ with respect to $\hat{s}_{\lambda,\varepsilon}$ -- one works in this case with the Gauduchon metric $\Omega = \sqrt[d-1]{\hat{s}_{\lambda,\varepsilon}}$ on $\hat{X}$. By \cref{prop:BG-Higgs} one gets a Bogomolov inequality for $\hat{E}$ with respect to $\hat{s}'_{\lambda,\varepsilon}$, which further gives the desired Bogomolov inequality for $E$.
\end{proof}

\section{Appendix: positive cones in K\"ahler geometry}\label{sect:appendix}

We give here some explanations and comments around Proposition \ref{prop:cones}. Throughout this appendix $(X,[\omega])$ will denote a polarized compact K\"ahler manifold of dimension $d$ with $\int_X\omega^d=1$ and $p$ will be an integer between $1$ and $d-1$.
\subsection{$d_{p,p}$ has closed range}
We follow the ideas of \cite[2]{HarveyLawson} for $p=1$ and their extension to arbitrary $p$ in \cite{AlessandriniAndreatta}, see also \cite{Alessandrini}. 
We denote the Fr\'echet spaces of complex differential $n$-forms or $(p,q)$-forms by $\cE^{n}(X)$ and $\cE^{p,q}(X)$ and by $\cE'_{n}(X)$ and $\cE'_{p,q}(X)$ their dual spaces of currents of dimension $n$ and bidimension $(p,q)$ on $X$ (endowed with their dual weak topology) respectively. A subscript $\R$ will indicate that we deal with real forms or currents. 

\begin{lemma}
    The restriction of the exterior differentiation operator 
    $$\d|_{\cE^{p,p}(X)_\R}:  \cE^{p,p}(X)_\R\to (\cE^{p+1,p}(X)\oplus\cE^{p,p+1}(X))_\R  $$ 
    has closed range. 
\end{lemma}

\begin{proof}
    The idea is to first show that the image of the above operator has finite codimension inside the subspace $\Ker(\d)$ of $\d$-closed forms inside  $(\cE^{p+1,p}(X)\oplus\cE^{p,p+1}(X))_\R$, see loc.cit.. Then a standard application of the Open Mapping Theorem shows the assertion. Indeed, if $j:L\to \Ker(\d)$ is the inclusion map of a (finite dimensional) algebraic complement to $\d(\cE^{p,p}(X)_\R)$ inside $\Ker(\d)$, then the operator $(\d,j):\cE^{p,p}(X)_\R\oplus L\to \Ker(\d)$ is surjective, hence open, and the conclusion easily follows.\end{proof}

We now look at the transposed operator
$$\d_{p,p}: (\cE'_{p+1,p}(X)\oplus\cE'_{p,p+1}(X))_\R \to \cE'_{p,p}(X)_\R$$
to
$$\d|_{\cE^{p,p}(X)_\R}:  \cE^{p,p}(X)_\R\to (\cE^{p+1,p}(X)\oplus\cE^{p,p+1}(X))_\R. $$
One has  $\d_{p,p}=\pi_{p,p}\circ\d|_{(\cE'_{p+1,p}(X)\oplus\cE'_{p,p+1}(X))_\R}$, where $\pi_{p,p}:\cE_{2p}'(X)_\R\to \cE'_{p,p}(X)_\R$ is the natural projection. Then by the Closed Range Theorem we get
\begin{corollary}\label{cor:ClosedRange}
  The operator $\d_{p,p}$ has closed range. In particular the natural projection map 
  $$\Ker(\d|_{\cE'_{p,p}(X)_\R}:\cE'_{p,p}(X)_\R\to \cE'_{2p-1}(X)_\R)\longrightarrow H^{d-p,d-p}_{BC}(X)_\R$$
  is continuous, where $H^{d-p,d-p}_{BC}(X)_\R$ is endowed with its separated linear topology. 
\end{corollary}

\subsection{The cone $\Pseff^p(X)$}

The subset $C$ of $\cE_{2p}'(X)$ consisting of closed (strongly) positive currents $T\in \cE_{2p}'(X)$ such that $\int_X T\wedge\omega^p=1$ is weakly compact. This is a consequence of the Banach-Alaoglu-Bourbaki Theorem, see 
\cite[Proposition III.1.23]{DemaillyBook}. Together with Corollary \ref{cor:ClosedRange} this gives

\begin{proposition}
The cone $\Pseff^{d-p}(X)$ is closed.
\end{proposition}

Note that until now the K\"ahler property has not been used in this section. 
It will be used in the next statement. (The chosen  positivity type of forms will not play any role in our statements as long as one considers the correct type for the dual cones.)

\begin{proposition}\label{prop:Pseff-full-dimensional}
If $X$ is K\"ahler, then  $\Pseff^{p}(X)$ is full dimensional and salient.
\end{proposition}

\begin{proof}
    If $(X,[\omega])$ is polarized K\"ahler, then clearly $[\omega^p]$ is a non-zero element in $\Pseff^{p}(X)$. There exists an open neighbourhood $V$ of $\omega^p$ in  $\Ker(\d|_{ \cE^{p,p}(X)_\R)}:\cE^{p,p}(X)_\R\to \cE^{2p+1}(X)_\R)$ consisting only of (strongly) positive $(p,p)$-forms. The natural projection $\Ker(\d|_{ \cE^{p,p}(X)_\R)}:\cE^{p,p}(X)_\R\to \cE^{2p+1}(X)_\R)\to H^{p,p}_{BC}(X)_\R$ is open and factors through $\Ker(\d|_{ \cE'_{d-p,d-p}(X)_\R)}:\cE'_{d-p,d-p}(X)_\R\to \cE'_{2d-2p-1}(X)_\R)$, hence the projection of $V$ to $H^{p,p}_{BC}(X)_\R$ is an open neighbourhood of $[\omega]$ lying inside $\Pseff^{p}(X)$. This shows that $\Pseff^{p}(X)$ is full dimensional.

    Suppose now that $[T]\in \Pseff^{p}(X)$ is the class of a closed positive current $T$ such that $-[T]\in\Pseff^{p}(X)$. Then we would have $0\le\int_XT\wedge\omega=-\int_X(-T)\wedge\omega\le0$ and thus $T$ must be zero. So $\Pseff^{p}(X)$ is salient. 
\end{proof}

\subsection{The cone $\Nef_A^p(X)$}
\begin{proposition}
    The cone $\Nef_A^p(X)$ is closed.
\end{proposition}
The proof goes exactly as in \cite[Lemma 2.3]{ChioseRasdeaconuSuvaina2019}, where the authors' restriction to the case $p\in\{1,d-1\}$ is not necessary.

\begin{proposition}
    If $X$ is K\"ahler, the cone $\Nef_A^p(X)$ is dual to
    ${\Pseff}^{d-p}(X)$. In particular it is full dimensional and salient.
\end{proposition}

\begin{proof}  

It is immediately seen that  
${\Nef}_{A}^p(X)\subset{\Pseff}^{d-p}(X)^\vee$.
We prove the opposite inclusion by adapting the proof of 
\cite[Lemme 1.3]{LamariJMPA} to our situation, where $p$ is arbitrary but $X$ is K\"ahler, (see also  \cite{ChioseRasdeaconuSuvaina2019} for the case $p=1$ in the balanced case).

Let $[\eta]\in H^{p,p}_A(X)_{\R}$ 
be a non-zero Aeppli cohomology class which is non-negative on ${\Pseff}^{d-p}(X)$, 
and let $\eta\in \cE^{p,p}(X)_{\R}$ be a representative of this class. We may and will assume that $\int_{X}\eta\wedge\omega^{d-p}=1$.

We put $K\subset\cE_{2p}'(X)$ to be the set consisting of (strongly) positive currents $T\in \cE_{2p}'(X)$ such that $\int_X T\wedge\omega^p=1$. This set is convex and weakly compact. Its intersection with $\Ker(\d|_{ \cE'_{p,p}(X)_\R}) $ will be denoted as before by $C$.

Since $[\eta]\ne0$ and $[\omega^{d-p}]$ lies in the interior of  ${\Pseff}^{d-p}(X)$ (by the proof of Proposition \ref{prop:Pseff-full-dimensional}), we have $\int_{X}\eta\wedge\omega^{d-p}>0.$

We fix some $\varepsilon>0$ and  set 
$K(\varepsilon):=K+\varepsilon \omega^{d-p}$ and 
$C(\varepsilon):=C+\varepsilon \omega^{d-p}$. We obviously have 
$C(\varepsilon)=K(\varepsilon)\cap \Ker(\d|_{ \cE'_{p,p}(X)_\R}) $
and
\begin{equation}\label{eq:positivity}
\int_{X}T\wedge\eta>0,  \ \forall T\in C(\varepsilon).
\end{equation}
The $(p,p)$-form $\eta$ defines a  continuous linear functional on $\Ker(\d|_{ \cE'_{p,p}(X)_\R}) $. We denote by $F$ its kernel. By the inequality \eqref{eq:positivity} we have
$$K(\varepsilon)\cap F=C(\varepsilon)\cap F=\emptyset.$$
Thus by Hahn-Banach there exists a $(p,p)$-form 
$\beta_{\varepsilon}$ which vanishes on $F$ and is strictly positive 
on $K(\varepsilon)$. 

Put   
$$\lambda_{\varepsilon}:=
\frac {\int_{X}\eta\wedge\omega^{d-p}}{\int_{X}\beta_{\varepsilon}\wedge\omega^{d-p}}.$$  Then
the $(p,p)$-form $\eta-\lambda_{\varepsilon}\beta_{\varepsilon}$ vanishes both on $F$ and on $\omega^{d-p}$, hence vanishes on their algebraic span which is  $\Ker(\d|_{ \cE'_{p,p}(X)_\R}).$
By
 the duality between 
$H^{p,p}_A(X)_{\R}$ and $H^{d-p,d-p}_{BC}(X)_{\R}$, 
it follows that there exists a $(p,p-1)$-form $\gamma_{\varepsilon}$ such that 
$$
\eta-\lambda_{\varepsilon}\beta_{\varepsilon}=
-\bar\partial\gamma_{\varepsilon}-\partial\bar\gamma_{\varepsilon}.
$$ 
Thus the $(p,p)$-form 
$$
\eta+\bar\partial\gamma_{\varepsilon}+
\partial\bar\gamma_{\varepsilon}=\lambda_{\varepsilon}\beta_{\varepsilon}
$$
is in the class $[\eta]\in H^{p,p}_A(X)_{\R}$ and 
is strictly positive on $K(\varepsilon)$. 

We will now show that 
$$\eta+\bar\partial\gamma_{\varepsilon}+
\partial\bar\gamma_{\varepsilon}\ge_{w}-\varepsilon\omega^{p}.$$
Let  $T\in K$. Then 
$$\int_{X}T\wedge(\eta+\bar\partial\gamma_{\varepsilon}+
\partial\bar\gamma_{\varepsilon}+\varepsilon\omega^{p})=
\int_{X}T\wedge(\eta+\bar\partial\gamma_{\varepsilon}+
\partial\bar\gamma_{\varepsilon})+\varepsilon=$$
$$\int_{X}T\wedge(\eta+\bar\partial\gamma_{\varepsilon}+
\partial\bar\gamma_{\varepsilon})+\int_{X}(\varepsilon\omega^{d-p})\wedge(\eta+\bar\partial\gamma_{\varepsilon}+
\partial\bar\gamma_{\varepsilon})=
\int_{X}(T+\varepsilon\omega^{d-p})\wedge(\eta+\bar\partial\gamma_{\varepsilon}+
\partial\bar\gamma_{\varepsilon})>0,$$
since $\eta+\bar\partial\gamma_{\varepsilon}+
\partial\bar\gamma_{\varepsilon}$ is strictly positive on $K(\varepsilon)$.
\end{proof}

\begin{remark}
    If $X$ is K\"ahler the interior of the nef cone inside $H^{p,p}_A(X)_{\R}$ is the cone of Aeppli cohomology classes $[\eta]$ represented by strictly weakly positive $\partial\bar\partial$-closed $(p,p)$-forms. Indeed, for a class $[\eta]$ in the interior of the Aeppli nef cone there exists some $\epsilon>0$ such that $[\eta-\epsilon\omega^p]$ is also nef, and therefore for some representative $\eta_{\epsilon/2}$ of $[\eta]$ we have $\eta_{\epsilon/2}-\epsilon\omega^p\ge_w-\frac{\epsilon}{2}\omega^p$, hence $\eta_{\epsilon/2}\ge_w\frac{\epsilon}{2}\omega^p>_w0$.
\end{remark}

\bibliography{bib2025Arxiv.bib}
\bibliographystyle{amsalpha}

\end{document}